\documentclass[12pt]{amsart}
\usepackage[margin=1.3in]{geometry} 

\usepackage{amsmath, amssymb, amsthm}
\usepackage{graphicx} 
\usepackage{mathrsfs}
\usepackage{mathtools}
\usepackage{xcolor}
\usepackage{ulem}

\usepackage{natbib}
 \usepackage{lineno}
\usepackage{hyperref}
\usepackage{tikz-cd}
\usepackage{yfonts}
\usepackage{bbm}

\newtheorem{theorem}{Theorem}[section]

\newtheorem{lemma}{Lemma}[section]
\newtheorem{example}{Example}[section]
\newtheorem*{remark}{Remark}

\newtheorem{proposition}{Proposition}[section]
\newtheorem*{proposition*}{Proposition}

\newcommand{\ad}{\mathrm{ad}}
\newcommand{\Vir}{\mathrm{Vir}}

\begin{document}

\title{Superconformal Algebras Over Arbitrary Rings of Coefficients}
\author{Consuelo Martínez}
\address{Departamento de Matem\'{a}ticas, Universidad de Oviedo, C/ Calvo Sotelo
s/n, Oviedo, 33007, Spain, and}
\address{Real Academia de Ciencias Exactas, Físicas y Naturales 
de España (RAC). C/ Valverde 22, 28004-Madrid, 
Spain.}
\email{cmartinez@uniovi.es}
\author{Efim Zelmanov}
\address{SICM, Southern University of Science and Technology, Shenzhen, 518055, China.}
\email{efim.zelmanov@gmail.com}
\author{Zezhou Zhang}
\address{Department of Mathematics, Beijing Normal University, Beijing, 100875, China.}
\email{zz2d@bnu.edu.cn}


\begin{abstract}
We define analogs of superconformal algebras over an arbitrary commutative associative superalgebra. In the finitely generated case the universal central extensions of these algebras are finitely presented. In particular, we show that all known superconformal algebras are finitely presented.	\end{abstract}

\setcounter{section}{-1}

\maketitle

\section{Introduction}

Let \( \mathbb{C}[t,t^{-1}] \) be the algebra of Laurent polynomials over the field of complex numbers \( \mathbb{C} \).

By the (centerless) Virasoro algebra, also known as Witt algebra, we mean the Lie algebra \( \mathrm{Vir} = \operatorname{Der} \mathbb{C}[t,t^{-1}] \) of all derivations of the algebra \( \mathbb{C}[t,t^{-1}] \).

The elements \( e_i = -t^{i+1} \frac{\partial}{\partial t} \), \( i \in \mathbb{Z} \), form a basis of the algebra \( \mathrm{Vir} \), where \( [e_i, e_j] = (i - j)e_{i+j} \). The central extension \( \widehat{\mathrm{Vir}} = \sum_{i \in \mathbb{Z}} \mathbb{C} e_i + \mathbb{C} c \), \( [e_i, c] = 0 \), \( [e_i, e_j] = (i - j)e_{i+j} + \delta_{{i+j},0} \frac{i^3-i}{12}c \), is called the Virasoro algebra.

A search for superconformal algebras, i.e., superextensions of the Virasoro algebra with similar properties, started in the works of A. Neveu and J.H. Schwarz \cite{neveu1971}, P. Ramond \cite{ramond1971}, followed by M. Ademolo et al. \cite{ademollo1976}, K. Schoutens \cite{schoutens1987}, A. Schwimmer and N. Seiberg \cite{schwimmer1987}. 

In \cite{kac1989}, V. Kac and J. van de Leur defined a superconformal algebra as:
\[
\text{a } \mathbb{Z} \text{-graded Lie superalgebra } L = L_{\bar{0}} \oplus L_{\bar{1}} = \sum_{i \in \mathbb{Z}} L_i, \text{ such that}
\]
\begin{enumerate}
\item \( L \) is simple,
\item the dimensions \( \dim_{\mathbb{C}} L_i \), \( i \in \mathbb{Z} \), are uniformly bounded,
\item \( \mathrm{Vir} \subseteq L_{\bar{0}} \).
\end{enumerate}


Let \( G(n) \) be the Grassmann algebra on \( n \) Grassmann variables:
\[
G(n) = \mathbb{C}\langle 1, \xi_1, \ldots, \xi_n \mid  \xi_i^2 = 0, \quad \xi_i \xi_j + \xi_j \xi_i = 0, \quad 1 \leqslant i, j \leqslant n\rangle.
\]

Consider the associative commutative superalgebra
\[
\Lambda(1:n) = \mathbb{C}[t, t^{-1}] \otimes_\mathbb{C} G(n).
\]

The only known superconformal algebras are:
\begin{enumerate}
\item \( W(1:n) = \operatorname{Der} \Lambda(1:n) \),
\item \( S(n, \gamma) \),
\item contact superalgebras \( K(1:n) \) of Ramond and Neveu–Schwarz types,
\item the exceptional superconformal algebra \( CK(6) \),
\item twisted contact superalgebras \( K^{(2)}(1:2n) \).
\end{enumerate}

The explicit constructions of these superalgebras in a more general context will be given in Chapter  {\ref{sec:BasicConstrut}}. 

V. Kac and J. van de Leur \cite{kac1989} conjectured that these are the only superconformal algebras. The results in \cite{fattori2002} and \cite{kac2001} partially confirm this conjecture. 


The purpose of this paper is to introduce constructions of superconformal algebras using an arbitrary finitely generated associative commutative superalgebra \( A \) over a field \( F \) of characteristic \( \neq 2 \) (instead of the Laurent polynomials \( \mathbb{C}[t, t^{-1}] \)) \textcolor{.}{ that contains \(\sqrt{-1}\). \footnote{The assumption \(\sqrt{-1} \in F \) only appears substantively in the construction of type \(K\).   }} We study basic algebraic properties of these superalgebras: finite generation and finite presentation.

In Section {2} we prove the following theorem.

\renewcommand{\thetheorem}{\Alph{theorem}}
\setcounter{theorem}{0}

\begin{theorem}\label{theorem:A}\mbox{} Let \(F\) be a field where ${char}(F)\neq 2$ and $\sqrt{-1}\in F$. Let \( A \) be an arbitrary associative commutative superalgebra. Let \(n\) be  an integer. Then the following hold: 

\begin{enumerate}
    \item For \( n \geqslant 2 \), the superalgebras \( W(A:n) \), \( S(\delta,n) \) are perfect; 

\noindent for \( n\geqslant 3 \), the superalgebras \( [K(A:n), K(A:n)] \) and \( [K^{(2)}(A:n), K^{(2)}(A:n)] \) are perfect; 

\noindent for an arbitrary even derivation of \( A \), the superalgebra \( CK(A,d) \) is perfect.

\noindent For a finitely generated superalgebra \( A \), these five kinds of (Lie) superalgebras are finitely generated.

\item If \( A \) is spanned by all values of all (even and odd) derivations, then the conclusions in item (1) hold for the superalgebra \( W(A:1) \). 

\item If \( A \) is a contact algebra with the derivation
\([a, 1] = d(a) \) and
\[
A = \{[a, b] + a \cdot d(b) \mid a, b \in A\},
\]
then the conclusions in item (1)  also hold for the superalgebra \( K(A:2) \).

\item If \( A = A^{(0)} \oplus A^{(1)} \) is a commutative associative superalgebra that is also a \( \mathbb{Z}/2\mathbb{Z} \)-graded contact superalgebra, with the derivation
\([a, 1] = d(a) \) (see section 1.4) and
\[
A = \{[a, b] + a \cdot d(b) \mid a, b \in A\},
\]
then certain assumptions (specified in Theorem \ref{thm:AforKtwisted}) on \( A^{(0)}\) and \( A^{(1)}\) allow 
the conclusions in item (1) to hold for  \( K^{(2)}(A:2) \).


\end{enumerate}

\end{theorem}


In Section {3} we discuss universal central extensions of Lie superalgebras.

In Section {4}, Section {5} we show that for a finitely generated associative commutative superalgebra \( A \), the universal central extensions of all  superalgebras in Theorem \ref{theorem:A} are finitely presented.

\begin{theorem}\label{theorem:B}
Let \( A \) be a finitely generated associative commutative superalgebra over a ground field of characteristic \( \neq 2 \). Under the assumptions of Theorem  \ref{theorem:A}, the universal central extensions of superalgebras \( W(A:n) \), \( S(\delta,n)  \), \( [K(A:n), K(A:n)] \), \( [K^{(2)}(A:n), K^{(2)}(A:n)] \), \( CK(A,d) \) are finitely presented.

Separately, we show that for a field \( F \) of zero characteristic the superalgebra 
\( K^{(2)}(F[t, t^{-1}]:2) \) is finitely presented.
\end{theorem}

D.B. Fairlie, J. Myers and C.K. Zachos \cite{fairlie1988} proved that the Neveu–Schwarz and the Ramond superalgebras are finitely presented. V. Kac and J. van de Leur \cite{kac1989} computed the second cohomology spaces for all known superconformal algebras. In particular, they showed that these spaces are finite dimensional.

These results, together with  Theorem \ref{theorem:B} and Theorem \ref{thm5.13manuscript} imply:

\begin{theorem}\label{theorem:C}
All known (Kac-Van de Leur) superconformal algebras are finitely presented.
\end{theorem}

Consider the polynomial algebra \( F[p_1, q_1, \ldots, p_m, q_m] \) with the classical Poisson bracket
\[
[f(p_i, q_j), g(p_i, q_j)] = \sum_{i=1}^m \left( \frac{\partial f}{\partial p_i} \frac{\partial g}{\partial q_i} - \frac{\partial f}{\partial q_i} \frac{\partial g}{\partial p_i} \right).
\]
It gives rise to the Hamiltonian superalgebra
\[
H(2m, n) = K(F[p_1, q_1, \ldots, p_m, q_m]: n) = F[p_1, q_1, \ldots, p_m, q_m] \otimes G(n).
\]

In comparison, the polynomial algebra \( F[p_1, q_1, \ldots, p_m, q_m, t] \) with the contact bracket
\[
[f, g] = \sum_{i=1}^m \left( \frac{\partial f}{\partial p_i} \frac{\partial g}{\partial q_i} - \frac{\partial f}{\partial q_i} \frac{\partial g}{\partial p_i} \right) + \frac{t}{2} \left( \frac{\partial f}{\partial t} g - \frac{\partial g}{\partial t} f \right).
\] gives rise to the contact superalgebra
\[
K(2m+1, n) = K(F[p_1, q_1, \ldots, p_m, q_m, t]: n).
\]

Theorem \ref{theorem:B} and the results of \cite{martinez2025} on cyclic homology imply:

\begin{theorem}\label{theorem:D}{~}
For arbitrary \( m, n \geqslant 1 \), the superalgebra \( H(2m, n) \) is finitely presented.
\end{theorem}
\renewcommand{\thetheorem}{\thesection.\arabic{theorem}}

\section{Basic Constructions}\label{sec:BasicConstrut}

\subsection{The Superalgebras \( W(A:n) \)}\mbox{}\\

From now on, all vector spaces are considered over a field \( F \) of characteristic \( \neq 2 \).

Let \( A \) be an associative commutative superalgebra with 1. The Grassmann algebra
\[
G(n) = F\langle 1, \xi_1, \ldots, \xi_n \mid \xi_i^2 = 0, \xi_i \xi_j + \xi_j \xi_i = 0, 1 \leqslant i, j \leqslant n \rangle
\]
is viewed as an associative commutative superalgebra \( G(n) = G(n)_{\bar{0}} \oplus G(n)_{\bar{1}} \), where
\[
\xi_{i_1} \cdots \xi_{i_k} \in G(n)_{\bar{\iota}} \quad \text{for } k \equiv \iota \pmod{2}.
\]

Consider the associative commutative superalgebra
\[
\Lambda = \Lambda(A:n) = A \bigotimes_F G(n).
\]
If \( n = 0 \) then \( \Lambda = A \).

Consider the Lie superalgebra
\[
W(A:n) = \operatorname{Der} \Lambda(A:n).
\]
It is easy to see that \(\Lambda \operatorname{Der} \Lambda \subset \operatorname{Der} \Lambda\). In addition, an arbitrary derivation \( d \in \operatorname{Der} A \) extends to a derivation of \( \Lambda \) via \( d(\xi_i) = 0 \), \( 1 \leqslant i \leqslant n \). Hence, \( \operatorname{Der}A \subseteq \operatorname{Der}\Lambda. \)

 We have
\[
L = \operatorname{Der} \Lambda = \Lambda \operatorname{Der} A + \sum_{i=1}^n \Lambda \frac{\partial}{\partial \xi_i}.
\]

If \( n \geqslant 2 \), \( 1 \leqslant i \neq j \leqslant n \), then the even derivations
\[
e = \xi_i \frac{\partial}{\partial \xi_j}, \quad f = \xi_j \frac{\partial}{\partial \xi_i}, \quad h = \xi_i \frac{\partial}{\partial \xi_i} - \xi_j \frac{\partial}{\partial \xi_j}
\]
form an \( \mathfrak{sl}(2) \)-triple, i.e., \( [e,f] = h \), \( [h,e] = 2e \), \( [h,f] = -2f \).

We call the abelian subalgebra
\[
H = \operatorname{Span}_F \left( \xi_i \frac{\partial}{\partial \xi_i} - \xi_j \frac{\partial}{\partial \xi_j} : 1 \leqslant i,j \leqslant n \right)
\]
the Cartan subalgebra of \( L \). The superalgebra \( L \) decomposes as a direct sum of eigenspaces relative to the action of \( H \):
\[
L = L_0 + \sum_{0 \neq \alpha \in H^*} L_{\alpha},
\]
\[ \mbox{ where }
L_0 = C_L(H) = {\color{.}(F\cdot 1 + {\xi}_1 \cdots {\xi}_n)\operatorname{Der} A} + \sum_{i=1}^{n} A \xi_i \frac{\partial}{\partial \xi_i},
\]
and the subspace \( L_{\alpha} \) is the eigenspace that corresponds to the eigenfunctional \( 0 \neq \alpha \in H^* \). We call the set
\[
\Delta = \{ 0 \neq \alpha \in H^* \mid L_{\alpha} \neq (0) \}
\]
the \textit{root system} of the superalgebra \( L \).

Let \( \widetilde{H} = \operatorname{Span}_F (\xi_i \frac{\partial}{\partial \xi_i} \mid 1 \leqslant i \leqslant n) \), \( H \subset \widetilde{H} \).

Let \( \widetilde{\omega}_i \in \widetilde{H}^* \) be the linear functional such that
\[
\langle \widetilde{\omega}_i, \xi_j \frac{\partial}{\partial \xi_j} \rangle = \delta_{ij}, \quad 1 \leqslant i,j \leqslant n.
\]

Let \( \omega_i \) be the restriction of the linear functional \( \widetilde{\omega}_i \) to \( H \). Then
\begin{align*}
\Delta =& \{ \omega_{i_1} + \cdots + \omega_{i_k}, 1 \leqslant i_1 < \cdots < i_k \leqslant n, 1 \leqslant k ~{\color{.}\lneq}~ n \}\\
        &\bigcup \{ \omega_{i_1} + \cdots + \omega_{i_k} - \omega_j, 1 \leqslant i_1 < \cdots < i_k \leqslant n, j \notin \{ i_1, \ldots, i_k \}, k \geqslant 0 \},\\
&L_{\omega_{i_1} + \cdots + \omega_{i_k}} = \xi_{i_1} \cdots \xi_{i_k} \operatorname{Der} A + \sum_{i=1}^n A \xi_{i_1} \cdots \xi_{i_k} \xi_i \frac{\partial}{\partial \xi_i},\\
&L_{\omega_{i_1} + \cdots + \omega_{i_k} - \omega_j} = A \xi_{i_1} \cdots \xi_{i_k} \frac{\partial}{\partial \xi_j}.
\end{align*} 

For \( n = 1 \) the Lie algebra \( L = W(A:1) \) decomposes into the sum of eigenspaces relative to the action of the operator \( \operatorname{ad}(\xi_1 \frac{\partial}{\partial \xi_1}) \),
\[
L = A \frac{\partial}{\partial \xi_1} + (\operatorname{Der} A + A \xi_1 \frac{\partial}{\partial \xi_1}) + \xi_1 \operatorname{Der} A.
\]

\subsection{The Superalgebras \( S(\delta, n) \)}\label{subsec:BasicConstrut-S}\mbox{}\\

Let \( A \) be an associative commutative superalgebra. We view \( A \) as a module over the Lie superalgebra \( \operatorname{Der} A \), given by \( [d, a] = d(a) \), \( d \in \operatorname{Der} A \), \( a \in A \).

An \textbf{even} linear mapping \( \delta : \operatorname{Der} A \to A \) is called a \textit{divergence} if

1) \( \delta \) is a 1-cocycle, i.e. \( \delta([d_1, d_2]) = [\delta(d_1), d_2] + [d_1, \delta(d_2)] \) for arbitrary derivations \( d_1, d_2 \in \operatorname{Der} A \);

2) \( \delta(a d) = a \delta(d) + (-1)^{|a| \cdot |d|} d(a) \).

\begin{remark}
In \cite{kosmann2002} such an operator \( \delta \) is referred to as divergence with zero curvature.
\end{remark}

\begin{example}
Let \( \Lambda(m:n) = F[x_1, x_1^{-1}, \ldots, x_m, x_m^{-1}] \otimes G(n) \),
\[
d = \sum_{i=1}^m f_i \frac{\partial}{\partial x_i} + \sum_{j=1}^n h_j \frac{\partial}{\partial \xi_j}; \quad f_i, h_j \in \Lambda(m:n).
\]
Then
\[
\operatorname{div}(d) = \sum_{i=1}^m \frac{\partial f_i}{\partial x_i} + \sum_{j=1}^n (-1)^{|h_j|} \frac{\partial h_j}{\partial \xi_j}
\]
is a divergence.
\end{example}

\begin{example}
Let \( \delta : \operatorname{Der} A \to A \) be a divergence. Fix an element \( a \in A_{\bar{0}} \). Then \( \delta_1(d) = d(a) + \delta(d) \) is a divergence as well.
\end{example}

\begin{example}
Let \( \delta : \operatorname{Der} A \to A \) be a divergence. Fix an invertible element \( b \in A_{\bar{0}} \). Then
\[
\delta_2(d) = \frac{1}{b} \delta(b d)
\]
is a divergence as well.
\end{example} 




Let \( (A, \delta) \) be an associative commutative superalgebra with divergence. Consider the superalgebra
\[
\Lambda = \Lambda(A:n) = A \otimes_F G(n), \quad n \geqslant 2.
\]
As we have seen above, \( \operatorname{Der} \Lambda = G(n) \operatorname{Der} A + \sum_{i=1}^n \Lambda \frac{\partial}{\partial \xi_i} \).  Define   
\[
\delta^{(n)}:   \sum_i u_i d_i + \sum_{j=1}^n a_j \frac{\partial}{\partial \xi_j} \longmapsto \sum u_i \delta(d_i) + \sum_{j=1}^n (-1)^{|a_j|} \frac{\partial a_j}{\partial \xi_j};
\]
where \( u_i \in G(n) \), \( d_i \in \operatorname{Der}A \), \( a_j \in \Lambda \). Then \( \delta^{(n)} \) is a divergence on \( \Lambda \).

Consider the Lie superalgebras
\[
\widetilde{S}(\delta, n) = \{ d \in \operatorname{Der} \Lambda \mid \delta^{(n)}(d) = 0 \},
\]
\[
S(\delta, n) = [\widetilde{S}(\delta, n), \widetilde{S}(\delta, n)].
\]

\begin{example}
    {\textcolor{.}{As a special case of Example 1.3,} let \( A = \mathbb{C}[t, t^{-1}] \), let \( \gamma \in \mathbb{C}^\times \). The divergence \( \delta \) is defined as  \( \delta(t^{i+1} \frac{\partial}{\partial t}) = (\gamma + i + 1)t^i \),  \(i~\in  \mathbb{Z} \). The superalgebras \( S(\delta, n) \) were introduced in \cite{ademollo1976,schwimmer1987} under the name ''\( SU_2 \)-superconformal algebras".}
\end{example}


The Cartan subalgebra \( H = \operatorname{Span}_F (\xi_i \frac{\partial}{\partial \xi_i} - \xi_j \frac{\partial}{\partial \xi_j} \mid 1 \leqslant i, j \leqslant n) \) lies in \( S = S(\delta, n) \). Hence, the root decomposition of \( S \) (resp. \(\widetilde{S}(\delta,n)\)) with respect to the action of \( H \) is inherited from the root decomposition of the Lie superalgebra \( W(A:n) \):
\[
S = S_0 + \sum_{\alpha \in \Delta} S_{\alpha}.
\]

The root system is the same as in the case of \( W(A:n) \). {\color{.} Given the form of \(H\), it is convenient to note} that \( \Delta \subseteq (\bigoplus_{i=1}^{n} \mathbb{Z} \omega_i) / \mathbb{Z}(\omega_1 + \cdots + \omega_n) \).
{\color{.} Specifically, \(S\) decomposes into the following subspaces:}
\[
S_{(\omega_{i_1} + \cdots + \omega_{i_r} - \omega_j)} = A \xi_{i_1} \cdots \xi_{i_r} \frac{\partial}{\partial \xi_j},
\]
\[
S_{(\omega_{i_1} + \cdots + \omega_{i_r})} = \]
\[ \{ \xi_{i_1} \cdots \xi_{i_r}  ( d + \sum_{\color{.} j \notin \{i_1,\ldots ,i_r\} } a_j \xi_j \frac{\partial}{\partial \xi_j}  )  |~   {\color{.}a_j \in A, \mbox{ }} d \in \operatorname{Der} A, \mbox{ } {\color{.}\delta(d) - \sum_{j \notin \{i_1,\ldots, i_r\}} a_j = 0}  \}, 
\]

\textcolor{.}{Similarly, \(\widetilde{S}(\delta,n) = \widetilde{S}_0 + \sum_{\alpha \in \Delta} \widetilde{S}_{\alpha}\). Note that \(\widetilde{S}_{\alpha}={S}_{\alpha}\), and} \[
\widetilde{S}_0 =\] \[\{ d_1 + \xi_1 \cdots \xi_n d_2 + \sum_{j=1}^n a_j \xi_j \frac{\partial}{\partial \xi_j} |~ a_j \in A,\mbox{ } d_1, d_2 \in \operatorname{Der} A, \mbox{ } {\color{.}\delta(d_1) = \sum_{j=1}^n a_j},\mbox{ } \delta(d_2) = 0 \}.
\]

We will also use
\[
\tag{1} A (\xi_i \frac{\partial}{\partial \xi_i} - \xi_j \frac{\partial}{\partial \xi_j}) \subseteq [S_{\omega_i-\omega_j}, S_{\omega_j-\omega_i}],
\]
\[
\tag{2}d_1 + \sum_{k=1}^n a_k \xi_k \frac{\partial}{\partial \xi_k} \in [S_{\omega_i-\omega_j}, S_{\omega_j-\omega_i}] + \sum_{i=1}^n [S_{\omega_i}, S_{-\omega_i}], \quad \delta(d_1) = {\color{.}\sum_{k=1}^n a_k},\] \[\textrm{ where }  a_k \in A ,  d_1 \in \operatorname{Der}A.
\]



\begin{align*}\tag{3}
 &\sum [S_{\omega_{i_1} + \cdots + \omega_{i_p}}, S_{\omega_{j_1} + \cdots + \omega_{j_q}}] \\ 
= \quad & \xi_1 \cdots \xi_n \cdot \operatorname{Span} \left( [d_1, d_2] - (-1)^{|d_1| \cdot |d_2|} \delta(d_2) d_1 + \delta(d_1) d_2 \middle| d_1, d_2 \in \operatorname{Der} A \right),
\end{align*}
where the summation runs over all divisions \( \{i_1, \ldots, i_p\} \sqcup \{j_1, \ldots, j_q\} = \{1, \ldots, n\} \), \( p+q=n \).



\subsection{Superalgebras \( K(A:n) \)}\label{subsec:BasicConstrut-K}\mbox{}\\

We start with definitions of Poisson and contact brackets.

Let \( A \) be an associative commutative superalgebra with 1. A graded bracket \( [\cdot, \cdot] : A \times A \to A \) is called a \textit{Poisson bracket} if

1) \( (A, [\cdot, \cdot]) \) is a Lie superalgebra,

2) \( [a b, c] = a [b, c] + (-1)^{|b| \cdot |c|} [a, c] b \) for arbitrary elements \( a, b, c \in A_{\bar{0}} \cup A_{\bar{1}} \); \( |b| \) is the parity of an element \( b \).

We say that a (super)algebra \( A \) with a Poisson bracket is a \textit{Poisson (super)algebra}.

A bracket \( [\cdot, \cdot] : A \times A \rightarrow A \) is called a \textit{contact bracket} if

1) \( (A, [\cdot, \cdot]) \) is a Lie superalgebra;

2) the linear mapping \( D(a) = [a, 1] \) is an even derivation of \( A \);

3) \( [ab, c] = a[b, c] + (-1)^{|b| \cdot |c|}[a, c]b + abD(c) \) for arbitrary elements \( a, b, c \in A_{\bar{0}} \cup A_{\bar{1}} \).

We call a (super)algebra \( A \) with a contact bracket a \textit{contact (super)algebra}.




\begin{proposition}\label{prop1}
\mbox{}

{\color{.} \begin{itemize}
    \item Let \( A \) be a contact superalgebra with a contact bracket \( [\cdot, \cdot] \) and the derivation \( D(a) = [a, 1] \), \( a \in A \);

    \item Let \( B = \sum_{i \geqslant 0} B_i \) be a graded Poisson (super)algebra. For an element \( b \in B_s \) we denote \( \deg(b) = s \). Abusing notation,  we denote the Poisson bracket on \( B \) by the same symbol \( [\cdot, \cdot] \);

    \item Suppose that there exists an integer \( k \neq 0 \) such that
\[
[B_i, B_j] \subseteq B_{i+j-k}
\]
for all \( i,j \geqslant 0 \), and the integer \( k \) is not equal to 0 in the field \( F \).
\end{itemize}}
Under these three assumptions, \textcolor{.}{with \( a_i \in A_{\bar{0}} \cup A_{\bar{1}} \), \( b_i \in B_{\bar{0}} \cup B_{\bar{1}} \) one can  define a contact} bracket on \( A \otimes_F B \) \textcolor{.}{ by: }
\[
[a_1 \otimes b_1, a_2 \otimes b_2] = (-1)^{|b_1| \cdot |a_2|} [a_1, a_2] \otimes b_1 b_2
\]
\[
+ (-1)^{|b_1| \cdot |a_2|} a_1 a_2 \otimes [b_1, b_2] \tag{5}
\]
\[
 +(-1)^{|b_1| \cdot |a_2|} \frac{1}{k} (\deg(b_1) a_1 D(a_2) - \deg(b_2) D(a_1) a_2) \otimes b_1 b_2, 
\]
\end{proposition}

\begin{remark}
    A particular case of this formula is contained in \cite[Cor.6.2]{SOLARTE2016291}.
\end{remark}

\begin{proof}
A straightforward computation.
\end{proof}


The \textit{Grassmann superalgebra} is an associative commutative graded superalgebra
\[
G(n) = \sum_{i \geqslant 0} G(n)_i
\]
with the Poisson bracket
\[
\{f, g\} = (-1)^{|f|} \sum_{i=1}^{n} \frac{\partial f}{\partial \xi_i} \frac{\partial g}{\partial \xi_i},
\]
\[
[G(n)_i, G(n)_j] \subseteq G(n)_{i+j-2}.
\]

For an arbitrary contact superalgebra \( A \) we consider the tensor product \( A \otimes_F G(n) \) with the contact bracket given in  \( (5) \).

This gives rise to the superalgebra
\[
K(A:n) = (A \otimes_F G(n), [\cdot, \cdot]).
\]

Choosing \( A = \mathbb{C}[t,t^{-1}] \) with the bracket
\[
\{f, g\} = f'g - fg' \quad \text{or} \quad \{f, g\} = t(f'g - fg')
\]
we obtain the superalgebras \( K(1:n) \) of Neveu-Schwarz or Ramond types known as 
``\( SO_n \)-superconformal algebras", which were introduced in \cite{ademollo1976}\cite{schoutens1987}.

\begin{remark}
    The class of superalgebras \( K(A:n) \) includes Hamiltonian superalgebras \( H(2k:n) \) (see \cite{kac1977}).
\end{remark}

\textcolor{.}{Recall that \(\sqrt{-1}\in{F}\). The parity of \(n\) is required for the following structural description of \(K(A:n)\):} 

 {\textit{Let the integer \( n = 2m \) be even.}} Up to a change of variables, we can assume that the Grassmann variables are \( \zeta_1, \ldots, \zeta_m, \eta_1, \ldots, \eta_m \) 
that 
\[ [\zeta_i, \zeta_j] = [\eta_i, \eta_j] = 0, \quad 
[\zeta_i, \eta_j] = \delta_{ij}, \quad 1 \leqslant i,j \leqslant m.
\]


We call the subalgebra
\[
H = \operatorname{Span}_F(\zeta_i \eta_i, 1 \leqslant i \leqslant m)
\]
the Cartan subalgebra of \( K(A:2m) \). 

Consider the linear functionals \( \omega_i \in H^* \), 
defined by \(\langle \omega_i, \zeta_j \eta_j \rangle = \delta_{ij}\). The action of \( \operatorname{ad}(H) \) on \( K(A:n) \) defines the root decomposition
\[
K(A:2m) = K(A:2m)_0 + \sum_{\alpha \in \Delta} K(A:2m)_{\alpha},
\]
\[
\Delta = \{ \omega_{i_1} + \cdots + \omega_{i_p} - \omega_{j_1} - \cdots - \omega_{j_q} \mid 1 \leqslant i_1 < \cdots < i_p \leqslant m;
\]
\[
1 \leqslant j_1 < \cdots < j_q \leqslant m; \{ i_1, \ldots, i_p \} \cap \{ j_1, \ldots, j_q \} = \emptyset;
\]
\[
p \geqslant 0, q \geqslant 0, p + q \geqslant 1 \};
\] the centralizer \( K(A:n)_0 \) of \( H \) in \( K(A:n) \) is \[ K(A:2m)_0 = A + \sum_{r\geqslant 1} A \zeta_{i_1} \eta_{i_1} \cdots \zeta_{i_r} \eta_{i_r} .\]


\vskip .3cm 

 {\textit{\textcolor{.}{As for an odd integer \( n=2m+1 \geqslant 3 \),}}} we change from the \textcolor{.}{classical} Grassmann variables \( \xi_1, \ldots, \xi_n \) to the set of variables \( \zeta_1, \ldots, \zeta_m, \eta_1, \ldots, \eta_m, \textcolor{.}{\xi_n'} \), where \( [\zeta_i, \zeta_j] = [\eta_i, \eta_j] = 0 \), \( [\zeta_i, \eta_j] = \delta_{ij} \), \( [\zeta_i, \textcolor{.}{\xi_n'}] = [\eta_i, \textcolor{.}{\xi_n'}] = 0 \); \( 1 \leqslant i,j \leqslant m \); \( [\textcolor{.}{\xi_n'}, \textcolor{.}{\xi_n'}] = 1 \).


As above, we call the subalgebra
\[
H = \operatorname{Span}_F(\zeta_i \eta_i, 1 \leqslant i \leqslant m)
\]
the Cartan subalgebra of \( K(A:n) \). The action of \( \operatorname{ad}(H) \) defines the root decomposition
\[
K(A:n) = K(A:n)_0 + \sum_{\alpha \in \Delta} K(A:n)_{\alpha}
\]
with the same root system \( \Delta \subset \bigoplus_{i=1}^m \mathbb{Z} \omega_i \)
as in the even case.

\subsection{Twisted Subalgebras \( K^{(2)}(A:n) \), \( n \geqslant 2 \)} \label{subsec:BasicConstrut-K twisted} \mbox{}\\ 

Let \( A \) be an associative commutative superalgebra with a contact bracket \( [\cdot, \cdot] \). Let \( A = A^{(0)} \oplus A^{(1)} \) be a \( \mathbb{Z}/2\mathbb{Z} \)-grading of \( A \) as a contact algebra, i.e. \( A^{(i)} A^{(j)} \subseteq A^{(i+j \bmod 2)} \),
\( [A^{(i)}, A^{(j)}] \subseteq A^{(i+j \bmod 2)} \).

We assume that this grading is compatible with the grading \( A = A_{\bar{0}} \oplus A_{\bar{1}} \), \textcolor{.}{i.e. \(A^{(i)}=A^{(i)}_{\bar{0}} \oplus A^{(i)}_{\bar{1}}\).}

\begin{example}\label{eg: grading on F[t,t^-1]}
Let \( A = F[t, t^{-1}] = F[t^2, t^{-2}] \oplus tF[t^2, t^{-2}] \). This is a \( \mathbb{Z}/2\mathbb{Z} \)-grading for the contact bracket
\[
[t^i, t^j] = (i-j)t^{i+j}
\]
of Ramond type, but not for the contact bracket \( [t^i, t^j] = (i-j)t^{i+j-1} \) of Neveu-Schwarz type.
\end{example}

Let \( \xi_1, \ldots, \xi_n \) be Grassmann variables of \( G(n) \), \( \xi_i \xi_j + \xi_j \xi_i = 0 \), \( 1 \leqslant i,j \leqslant n \). Let \( G(n-1) = \langle 1, \xi_1, \ldots, \xi_{n-1} \rangle \). Then
\[
K(A:n) = (A^{(0)} \otimes G(n-1) + A^{(1)} \otimes \xi_n G(n-1)) \oplus (A^{(1)} \otimes G(n-1) + A^{(0)} \otimes \xi_n G(n-1))
\]
is a \( \mathbb{Z}/2\mathbb{Z} \)-grading of the Lie superalgebra \( K(A:n) \).

The Lie superalgebra
\[
K^{(2)}(A:n) = A^{(0)} \otimes G(n-1) + A^{(1)} \otimes \xi_n G(n-1)
\]
is a \textit{twisted superalgebra of the series \( K \)}.

\subsection{Exceptional Cheng-Kac Superalgebras \( CK(A,d) \)}\label{subsec:CK construction} \mbox{}\\

In \cite{martinez2001}, a Lie superalgebra \( CK(A,d) \) was constructed for an arbitrary associative commutative superalgebra \( A \) with an \textbf{even} derivation \( d \). Choosing \( A = \mathbb{C}[t, t^{-1}] \), \( d = \frac{\partial}{\partial t} \) or \( t \frac{\partial}{\partial t} \), one \textcolor{.}{ recovers} the superalgebra \( CK(6) \). We briefly recall the construction of \( CK(A,d) \) following \cite{martinez2001,martinez2003}:



The simple finite dimensional Lie superalgebra \( P(n-1) \) is the superalgebra of \( 2n \times 2n \) matrices of the type
\[
\begin{pmatrix}
a & k \\
h & -a^T
\end{pmatrix},
\]
where \( a, k, h \) are \( n \times n \) matrices over \( F \),  \(
\operatorname{tr}(a) = 0\), \( k^T = -k\),  \(h^T = h.\)
The superalgebras \( P(n) \) (the so called ``strange" series), \( n \geqslant 4 \), are centrally closed. However, \( P(3) \) has a nontrivial universal central extension \( \widehat{P(3)} \) (see \cite{martinez2003}). Its existence is related to the fact that the Lie algebra \( \mathrm{Skew}_4(F) \) of skew-symmetric \( 4 \times 4 \) matrices is a direct sum of two ideals,
\[
\mathrm{Skew}_4(F) = \mathfrak{sl}(2) \oplus \mathfrak{sl}(2).
\]



For an arbitrary element \( k \in \mathrm{Skew}_4(F) \) we consider the decomposition \( k = k_+ + k_- \) and let
\[
\varphi(k) = k_+ - k_-.
\]
The universal central extension  \( \widehat{P(3)} \) of \( P(3) \) can be realized as the superalgebra of \( 8 \times 8 \) matrices over the polynomial algebra \( F[d] \) of the type
\[
\begin{pmatrix}
a & k \\
\varphi(k)d + h & -a^T
\end{pmatrix} + \gamma d \cdot I_8,
\]
where \( a, k, h \) are \( 4 \times 4 \) matrices over \( F \), \( \operatorname{tr}(a) = 0 \), \(k^T = -k\),  \(h^T = h\),  \(\gamma \in F\) \text{and}   \(I_8\) \text{is the identity} \(8 \times 8\) \text{matrix, see \cite{martinez2003}}.

\textcolor{.}{Recall that} \(A\) is an associative commutative superalgebra with an even derivation \( d \). Consider the Weyl algebra \( W = \sum_{i \geqslant 0} A d^i \), where the variable \( d \) does not commute with coefficients \( a \in A \), but \( da = ad + d(a) \).

The superalgebra \( CK(A,d) \) is the subsuperalgebra of the Lie superalgebra \( \mathfrak{gl}_8(W) \) of \( 8 \times 8 \) matrices over \( W \) generated by \( \widehat{P(3)} \) and by all matrices 
\[
\begin{pmatrix} e_{ij}(a) & 0 \\ 0 & -e_{ji}(a) \end{pmatrix}, \quad a \in A, \quad 1 \leqslant i \neq j \leqslant 4,
\]
where \( e_{ij}(a) \) is the \( 4 \times 4 \) matrix having the element \( a \) at the intersection of the \( i \)-th row and the \( j \)-th column, and zeros everywhere else.

The Cartan subalgebra in this case is
\[
H = \{ \operatorname{diag}(\gamma_1, \gamma_2, \gamma_3, \gamma_4, -\gamma_1, -\gamma_2, -\gamma_3, -\gamma_4) \mid \gamma_i \in F, \mbox{ } \sum_{i=1}^4 \gamma_i = 0 \}.
\]

As above, define the functionals \( \omega_i \in H^* \),
\[
\langle \omega_i, \operatorname{diag}(\gamma_1, \gamma_2, \gamma_3, \gamma_4, -\gamma_1, -\gamma_2, -\gamma_3, -\gamma_4) \rangle = \gamma_i, \quad 1 \leqslant i \leqslant 4.
\]

\textcolor{.}{Abusing notation,} the dual space \( H^* \) will be identified with
\[
\bigoplus_{i=1}^4 F \omega_i / F(\omega_1 + \omega_2 + \omega_3 + \omega_4).
\]
The coset of \( \omega_i \) 
maps
\(
h = \operatorname{diag}(\gamma_1, \gamma_2, \gamma_3, \gamma_4, -\gamma_1, -\gamma_2, -\gamma_3, -\gamma_4)\) \text{ to } \(\gamma_i\) \text{ for } \(1 \leqslant i \leqslant 4\).

The even roots are \( \Delta_{\bar{0}} = \{ \omega_i - \omega_j \mid 1 \leqslant i \neq j \leqslant 4 \} \) and the odd roots \footnote{\textcolor{.}{We remind the reader that \(\omega_1 + \omega_2 + \omega_3 + \omega_4=0\) in \(H^*\)}} are \( \Delta_{\bar{1}} = \{ -\omega_i - \omega_j \mid 1 \leqslant i \leqslant j \leqslant 4 \} \), \(\Delta = \Delta_{\bar{0}} \sqcup \Delta_{\bar{1}}.
\)


Consider the following  elements of \( CK(A,d) \):
\[
e_{\omega_i - \omega_j}(a) = \begin{pmatrix} e_{ij}(a) & 0 \\ 0 & -e_{ji}(a) \end{pmatrix},
\]
\[
h_{\omega_i - \omega_j}(a) = \begin{pmatrix} e_{ii}(a) - e_{jj}(a) & 0 \\ 0 & e_{jj}(a) - e_{ii}(a) \end{pmatrix},
\]
\[
q_{-\omega_i - \omega_j}^{(-)}  = \begin{pmatrix} 0 & 0 \\ e_{ij}(1) + e_{ji}(1) & 0 \end{pmatrix}, \quad 1 \leqslant i,j \leqslant 4.
\]

Also define
\[
q_{-\omega_i - \omega_j}^{(-)}(a) = [q_{-\omega_i - \omega_k}^{(-)}, e_{\omega_k - \omega_j}(a)], \quad 1 \leqslant i,j,k \leqslant 4; \quad k \neq i,j, \quad a \in A.
\]


This definition does not depend on a choice of \( k \). In particular, \( q_{-2\omega_i}(a) = [q_{-\omega_i - \omega_k}^{(-)}, e_{\omega_k - \omega_i}(a)] \), \( k \neq i \).

For {\color{.} \( i \neq j \)} we also define the element
\[
q_{\omega_i + \omega_j}^{(+)} = \begin{pmatrix} 0 & e_{ij}(1) - e_{ji}(1) \\ \varphi(e_{ij}(1) - e_{ji}(1)){\color{.}d} & 0 \end{pmatrix}.
\]

In \cite{martinez2003} it was shown that
\[
CK(A,d)_{\omega_i - \omega_j} = e_{\omega_i - \omega_j}(A), \quad 1 \leqslant i \neq j \leqslant 4;
\]
\[
CK(A,d)_{-2\omega_i} = q_{-2\omega_i}(A);
\]
\[
CK(A,d)_{\omega_i + \omega_j} = CK(A,d)_{-\omega_k - \omega_\ell} = [q_{\omega_i + \omega_k}^{(+)}, e_{\omega_j - \omega_k}(A)] + q_{-\omega_k - \omega_\ell}^{(-)}(A),
\]
where \( \{i, j, k, \ell\} = \{1, 2, 3, 4\} \).


For an arbitrary element \( a \in A \) consider the elements
\begin{align*}
\mathrm{Vir}(a) = &\mbox{ }[[e_{\omega_4 - \omega_1}(a), q_{\omega_3 + \omega_1}^{(+)}], q_{\omega_2 + \omega_1}^{(+)}]. 
\end{align*} \textcolor{.}{Let \(\textrm{Vir}(A)=\operatorname{Span}_F \{\textrm{Vir}(a) \mid a \in A \}\);  \textrm{Vir}(A) is a matrix Lie superalgebra.}

\textcolor{.}{We give an explicit formula for \(a\) even:}
\begin{align*}
    \mathrm{Vir}(a) = & \mbox{ }I_8 (a d) - \begin{pmatrix} e_{11}(a') & 0 \\ 0 & -e_{11}(a') + I_4(a') \end{pmatrix},
\end{align*}
where \( a' = [a, d] = {\color{.} -d(a)} \), \( I_4(a) = \operatorname{diag}(a, a, a, a) \). 



The mapping
\[
a d \mapsto \mathrm{Vir}(a)
\]
is an embedding \textcolor{.}{(of Lie superalgebras)}. In \cite{martinez2003} it was shown that \textcolor{.}{the centralizer of \(H\) in \(CK(A,d)\) is }
\[
C_{CK(A,d)}(H) = \sum_{1 \leqslant i \neq j \leqslant 4} h_{\omega_i - \omega_j}(A) + \mathrm{Vir}(A) \cong A \otimes_F H + \mathrm{Vir}(A).
\] 


\section{Finite Generation}

\subsection{Superalgebras \( W(A:n) \)}\label{subsec:FinGen-W} \mbox{}\\

Let \( A \) be an associative commutative superalgebra. Let \( A^{\mathcal{D}} \) denote the linear span of all values of all (even and odd) derivations of the superalgebra \( A \),
\[
A^{\mathcal{D}} = \operatorname{Span}_F \{ d(a) \mid d \in \operatorname{Der} A, a \in A \}.
\]
Clearly, \( A^{\mathcal{D}} \) is an ideal of \( A \).

Recall that a Lie (super)algebra \( L \) is \textit{perfect} if \( L = [L, L] \).

We will prove Theorem \ref{theorem:A}  for the class \(W.\)

\begin{theorem}
\mbox{}

(1) The Lie superalgebras \( W(A:n), n \geqslant 2 \), are perfect. If the superalgebra \( A \) is finitely generated then so is the Lie superalgebra \( W(A:n) \).

(2) The Lie superalgebra \( W(A:1) \) is perfect if and only if \( A^{\mathcal{D}} = A \). If the superalgebra \( A \) is finitely generated then the Lie superalgebra \( W(A:1) \) is finitely generated if and only if \( A^{\mathcal{D}} \) \textcolor{.}{(as an \(F\)-submodule)} has finite codimension in \( A \).
\end{theorem}

\begin{proof}
Let \( a', a'' \in \Lambda(A:n) = A \otimes G(n) \) \textcolor{.}{ be odd or even}. Suppose that \( n \geqslant 2 \). If the element \( a' \) does not involve \( \xi_2 \) and the element \( a'' \) does not involve \( \xi_1 \), then
\[
[a' \frac{\partial}{\partial \xi_1}, a'' \xi_1 \frac{\partial}{\partial \xi_2}] = (-1)^{|a''|} a' a'' \frac{\partial}{\partial \xi_2}. \tag{6}
\]

If an element \( a' \) does not involve \( \xi_2 \), \( a'' \) is an arbitrary element from \( \Lambda(A:n) \) and \( d \in \operatorname{Der} A \), then
\[
[a' \xi_2 d, a'' \frac{\partial}{\partial \xi_2}] = \pm {a''} {a'} d \pm {a'} \xi_2 d(a'') \frac{\partial}{\partial \xi_2}. \tag{7}
\]

The equality (6) implies the inclusion
\[
\sum_{i=1}^n \Lambda(A:n) \frac{\partial}{\partial \xi_i} \subseteq [W(A:n), W(A:n)].
\]

The equality (7) implies
\[
\Lambda(A:n) \operatorname{Der} A \subseteq [W(A:n), W(A:n)] + \sum_{i=1}^n \Lambda(A:n) \frac{\partial}{\partial \xi_i} = [W(A:n), W(A:n)].
\]

This proves that the superalgebra \( W(A:n), n \geqslant 2 \) is perfect.

Suppose that the superalgebra \( A \) is generated by elements \( a_1 = 1, a_2, \ldots, a_m \). In \cite{billig2018} Billig and Futorny showed that the Lie superalgebra \( \operatorname{Der} A \) is a finitely generated \( A \)-module. Let \( \operatorname{Der} A = \sum_{j=1}^k A d_j \). The equalities (6), (7) imply that the Lie superalgebra \( W(A:n) \) is generated by elements \( a_i \frac{\partial}{\partial \xi_k}, a_i \xi_j\frac{\partial}{\partial \xi_k}, \xi_j d_s \);  \( 1 \leqslant i \leqslant m, 1 \leqslant j,k \leqslant n, 1 \leqslant s \leqslant r \).

This completes the proof of the part (1) of the theorem.

Suppose now that \( n = 1 \). We have
\[
[W(A:1), W(A:1)] = A \frac{\partial}{\partial \xi_1} + (\operatorname{Der} A + A^{\mathcal{D}} \xi_1 \frac{\partial}{\partial \xi_1}) + \xi_1 \operatorname{Der} A.
\]

From this formula it follows that \( W(A:1)/[W(A:1), W(A:1)] \) and \( A/A^{\mathcal{D}} \) are isomorphic as vector spaces. Hence the Lie superalgebra \( W(A:1) \) is perfect if and only if \( A = A^{\mathcal{D}} \). 

Suppose that the superalgebra \( W(A:1) \) is finitely generated. If the superalgebra \( W(A:1) \) is finitely generated, then
\[
\dim_F W(A:1)/[W(A:1), W(A:1)] < \infty,
\]
which implies
\[
\dim_F A/A^{\mathcal{D}} < \infty.
\]

Now suppose that \( \dim_F A/A^{\mathcal{D}} < \infty \). To show that the Lie superalgebra \( W(A:1) \) is finitely generated, we need the equalities
\[
[a' \frac{\partial}{\partial \xi_1}, a'' \xi_1\frac{\partial}{\partial \xi_1}] = \pm a' a''\frac{\partial}{\partial \xi_1} \tag{9}
\]
\[
[a' \xi_1 d, a''\xi_1 \frac{\partial}{\partial \xi_1}] = \pm a' a'' \xi_1 d\tag{10}
\]
\[
[\frac{\partial}{\partial \xi_1}, a\xi_1 d] = \pm ad \tag{11}
\]
\[
[\xi_1 d, a \frac{\partial}{\partial \xi_1}] = \pm a d \pm \xi_1 d(a) \frac{\partial}{\partial \xi_1} \tag{12}
\]
for arbitrary elements \( a', a'', a \in A_{\bar{0}} \cup A_{\bar{1}} \), \( d \in (\operatorname{Der} A)_{\bar{0}} \cup (\operatorname{Der} A)_{\bar{1}} \).

As above, let \( \operatorname{Der} A = \sum_{j=1}^k A d_j \). Suppose that the superalgebra \( A \) is generated by elements \( a_1 = 1, a_2, \ldots, a_m \) and, according to our assumption,
\[
A = \sum_{i=1}^s F u_i + \sum_{j=1}^k A d_j(A).
\]

We claim that the Lie superalgebra \( W(A:1) \) is generated by elements \[\frac{\partial}{\partial \xi_1}, \mbox{ } a_i\xi_{1}\frac{\partial}{\partial \xi_1},\mbox{ } \xi_1 d_j ,\mbox{ } u_k \xi_1 \frac{\partial}{\partial \xi_1}, \quad   1 \leqslant i \leqslant m, 1 \leqslant j \leqslant r, 1 \leqslant k \leqslant s, \]

Let \( W' \) be the Lie superalgebra of \( W(A:1) \) generated by these elements. The equality (9) implies that \( A \frac{\partial}{\partial \xi_1} \subseteq W' \). The equality (10) implies \( \xi_1 \operatorname{Der}A \subseteq W' \). Since \( \operatorname{Der} A = [\frac{\partial}{\partial \xi_1}, \xi_1 \operatorname{Der}A] \), by (11) we conclude that \( \operatorname{Der} A \subseteq W' \). Equality (12) implies that
\[
A^{\mathcal{D}} \xi_1 \frac{\partial}{\partial \xi_1} \subseteq [\xi_1 \operatorname{Der} A, A \frac{\partial}{\partial \xi_1}] + \operatorname{Der} A \subseteq W'
\]
and therefore
\[
A \xi_1 \frac{\partial}{\partial \xi_1} = \sum_{k=1}^{s} F u_k \xi_1 \frac{\partial}{\partial \xi_1} + A^{\mathcal{D}} \xi_1 \frac{\partial}{\partial \xi_1} \subseteq W'.
\]

We have proved that \( W' = W(A:1) \). This completes the proof of the theorem.
\end{proof}

\begin{remark}
    P. Etingof and T. Schedler (see \cite{etingof2016}) constructed an example of a finitely generated associative commutative algebra \( A \) such that \( \dim_F A/A^{\mathcal{D}} = \infty \).
\end{remark}


\subsection{Superalgebras \( S(\delta, n) \), \( n \geqslant 2 \)} \mbox{}\\

As above, we consider an associative commutative superalgebra \( A \) with a divergence \( \delta : \operatorname{Der} A \rightarrow A \):
and extend \( \delta \) to the divergence \( \delta^{(n)} : \operatorname{Der} \Lambda(A:n) \rightarrow \Lambda(A:n) \).

We denote
\[
\widetilde{S}(\delta, n) = \{ d \in W(A:n) \mid \delta^{(n)}(d) = 0 \},
\]
\[
S = S(\delta, n) = [\widetilde{S}(\delta, n), \widetilde{S}(\delta, n)].
\]

Denote also
\[
S^* = \sum_{\alpha \in \Delta} S_\alpha = \sum_{\alpha \in \Delta} \widetilde{S}(\delta, n)_\alpha.
\]

Let \( \alpha \in \Delta \). From
\(
\xi_i \frac{\partial}{\partial \xi_i} - \xi_j \frac{\partial}{\partial \xi_j} = [\xi_i \frac{\partial}{\partial \xi_j}, \xi_j \frac{\partial}{\partial \xi_i}]
\)
we conclude that \( S_\alpha \subseteq [S, S] \), \( S^* \subseteq [S, S] \).

Let us prove Theorem A for \textcolor{.}{class} \( S(\delta, n) \), \( n \geqslant 2 .\)

\begin{theorem}\mbox{}

(1) The Lie superalgebra \( S = S(\delta, n), n \geqslant 2 \), is perfect;

(2) if the superalgebra \( A \) is finitely generated then so is the superalgebra \( S \).
\end{theorem}

\begin{proof}
We will show that the superalgebra \( S \) is generated by \( S^* \), i.e. \( S = \langle S^* \rangle \). \textcolor{.}{ Formulas (1), (2) from section  {\ref{subsec:BasicConstrut-S}}.} imply that
\[
\widetilde{S}(\delta, n) = \langle S^* \rangle + \xi_1 \cdots \xi_n \cdot S(\delta),
\]
where \( S(\delta) = \{ d \in \operatorname{Der} A \mid \delta(d) = 0 \} \).

Now,
\[
S(\delta, n) = [\widetilde{S}(\delta, n), \widetilde{S}(\delta, n)] = \langle S^* \rangle + \xi_1 \cdots \xi_n [S(\delta), S(\delta)].
\]

Formulas (3), (4) imply that
\[
\xi_1 \cdots \xi_n [S(\delta), S(\delta)] \subseteq [S^*, S^*].
\]

This proves the part (1) of the theorem.

The part (2) is proved similarly to the corresponding part of Theorem  \ref{theorem:A} for superalgebras \( W(A:n) \), see  Section { \ref{subsec:FinGen-W}} . This completes the proof of the theorem.
\end{proof}

\subsection{Superalgebras \( K(A:n) \), \( n \geqslant 2 \)}\label{subsec:FinGen-K(no.1)} \mbox{}\\

Let \( A \) be a contact superalgebra with derivation \( d(a) = [a, 1] \), \( a \in A \). As above,
\[
G(n) = \langle 1, \xi_1, \ldots, \xi_n \mid \xi_i^2 = 0, \xi_i \xi_j + \xi_j \xi_i = 0, 1 \leqslant i, j \leqslant n \rangle
\]
is the Grassmann superalgebra. In the Grassmann superalgebra \( G(n) \) we define the Poisson bracket \( [\cdot, \cdot] \) such that \( [\xi_i, \xi_j] = \delta_{ij} \), \( 1 \leqslant i, j \leqslant n \). \footnote{When \(\sqrt{-1} \in F\), change of variables equate this bracket to the one in Section {\ref{subsec:BasicConstrut-K}}.}

Now we assume that \( 0 \neq \sqrt{2} \in F, \sqrt{-1} \in F \), and form new variables 
\[
\zeta_i = \frac{1}{\sqrt{2}} (\xi_{2i-1} + \sqrt{-1} \xi_{2i}), \quad 
\eta_i = \frac{1}{\sqrt{2}} (\xi_{2i-1} - \sqrt{-1} \xi_{2i}), \quad 1 \leqslant i \leqslant m.
\]

\begin{itemize}
    \item For even \( n = 2m \), clearly  \( [\zeta_i, \zeta_j] = [\eta_i, \eta_j] = 0 \); \( [\zeta_i, \eta_j] = \delta_{ij} \), \( 1 \leqslant i, j \leqslant m \).
    \item For odd \( n = 2m+1 \), clearly \( \zeta_i, \eta_j, 1 \leqslant i,j \leqslant m \) bracket together as above. In addition, we set \( \mu = \xi_{2m+1} \).
\end{itemize} 

 For an arbitrary integer \( r \in [0,n]\), consider the subspace
\[
M_r = \operatorname{Span}_F \left\{ [a, b] + \frac{r}{2} a d(b) \middle| a, b \in A \right\}.
\]

For arbitrary elements \( a, b \in A_{\bar{0}} \cup A_{\bar{1}} \) we have
\[
\left( [a, b] + \frac{r}{2} a d(b) \right) + (-1)^{|a||b|} \left( [b, a] + \frac{r}{2} b d(a) \right) = \frac{r}{2} d(ab),
\]
which implies 
\[
d(A) \subseteq M_r.
\]


\subsection{Root Decomposition of \( K(A:n) \)}\label{subsec:FinGen-K(no.2)} \mbox{}\\

Let \( L = K(A:n) \). Assuming that \( n \geqslant 2 \) let
\[
L = L_0 + \sum_{\alpha \in \Delta} L_\alpha
\]
be the root decomposition, see Section {\ref{subsec:BasicConstrut-K}} .

Denote \( L^* = \sum_{\alpha \in \Delta} L_\alpha \). Let \( V = \sum_{i=1}^n F \xi_i \). 
\textcolor{.}{We stress that as in Section {\ref{subsec:FinGen-K(no.1)}}, there are two generating sets of variables for 
\(G(n)\): the set of \(\xi\) variables and the set of \(\{\zeta, \eta, \mu\}\) variables. Clearly both sets form \(F\)-bases for \( V\).}

\begin{lemma}\label{lem: KperfFirstLemma} Let \(V^{r}\) be the product of \(r\) copies of \(V\), then: 
\[
[L, L] = \sum_{r=0}^{n-1} A\cdot V^{ r} + M_n \xi_1 \cdots \xi_n. \tag{13}
\]
\end{lemma}


\begin{proof}
For a nonempty subset \( X = \{i_1 < \ldots < i_s\} \subseteq \{1, 2, \ldots, n\} ,  \) we define
\[
\xi_X = \xi_{i_1} \cdots \xi_{i_s}, \quad \xi_{\emptyset} = 1.
\]

For arbitrary elements \( a, b \in A \) we have
\[
[a \xi_X, b \xi_Y] = 
\begin{cases} 
\pm \left( [a, b] + \frac{1}{2} (|X| a \cdot d(b) - |Y| d(a) \cdot b) \right) \xi_{X{\cup}Y} & \text{if } X \cap Y = \emptyset; \\
\pm ab \xi_{Z}, \quad Z=(X\cup Y) \backslash (X \cap Y)   & \text{if } |X \cap Y| = 1; \\
0 & \text{if } |X \cap Y| \geqslant 2.
\end{cases} \tag{14}
\]

Let us show that \( [L, L] \) lies in the right hand side of (13): By (14) we need only to check that \( [a \xi_X, b \xi_Y] \subseteq M_n {V^n} \), assuming that \( X \cap Y = \emptyset \), \( X \cup Y = \{1, 2, \ldots, n\} \), or, equivalently,
\[
  [a, b] + \frac{1}{2} (|X| a \cdot d(b) - |Y| d(a) \cdot b) \in M_n.
\]

Indeed, since \(|X|+|Y|=n , \) it follows that the left hand side of the inclusion above is equal to \([a, b] + \frac{n}{2} a \cdot d(b) - \frac{|Y|}{2} d(ab).\) The latter element lies in \(M_n\) ( since  \(d(A) \subset M_n\) ).

Let us show that \( A \cdot  V^r \), \(  r < n \), and \( M_n {V^n} \) lie in \( [L, L] \):

 {\it Case 1.} Suppose that \( n = 2m \) is even.

Let \( r \) be odd, \( r < 2m \). We claim that
\[
A \cdot  V^r \subseteq L^* \subseteq [L, L]. \tag{15}
\]

Indeed, choose \( r \) different  elements \( u_1, \ldots u_r \in  \{\zeta_{1},  \eta_{1}, \ldots \zeta_{m},  \eta_{m}\} \) . Then \( A u_1 \cdots u_r \subseteq L_\alpha \), \( \alpha = \pm \omega_{i_1} + \cdots + \pm \omega_{i_r} \neq 0 \). \textcolor{.}{This proves (15) as the \(A\)-span of \(u_1\ldots u_r\) contains \(A \cdot  V^r \).}

\textcolor{.}{Let \( r \) be even, \( r < 2m \). Consider}  \( X \subsetneq \{1, 2, \ldots, n\} \) with even \( |X|=r \). Decompose \( X \) into a disjoint union of two subsets \( X', X'' \) (these subsets may be empty) such that \( |X'|, |X''| \) are even. Since \( X \neq \{1, 2, \ldots, n\} \) there exists \( i \in \{1, 2, \ldots, n\} \setminus X \). Choose \( a, b \in A \). Then
\[
a \xi_{X' \cup \{i\}}, b \xi_{X'' \cup \{i\}} \in L^*.
\]

By (14) we have
\[
[a \xi_{X' \cup \{i\}}, b \xi_{X'' \cup \{i\}}] = \pm ab \xi_X. \tag{16}
\]

This implies \( A \xi_X \subseteq [L^*, L^*] \) \textcolor{.}{and thus gives \(A \cdot  V^r  \subseteq [L, L]\) for all \(r<n=2m\).}

\textit{It remains to show that \( M_n \xi_{\{1, 2, \ldots, n\}} \subseteq [L, L] \):}

The set \( X = \{1, 2, \ldots, n\} \) can be decomposed into a disjoint union \( X = X' \cup X'' \), where both \( |X'| \) and \( |X''| \) are odd. Set \( a, b \in A_{\bar{0}} \cup A_{\bar{1}} \).

We have
\[
[a \xi_{X'}, b \xi_{X''}] = \pm \left( [a, b] + \frac{1}{2} (|X'| a d(b) - |X''| d(a) b) \right) \xi_X \in [L^*, L^*]. \tag{17}
\]

Choosing \( b = 1 \) we get
\[
\left(1 - \frac{|X''|}{2}\right) d(a) \xi_X \in [L^*, L^*],
\]
which implies \( d(A) \xi_X \subseteq [L^*, L^*] \).

Arguing as above, we get
\[
[a, b] + \frac{1}{2} (|X'| a \cdot d(b) - |X''| d(a) b) = [a, b] + \frac{n}{2} a \cdot d(b) - \frac{|X''|}{2} d(ab),
\]
which implies \( M_n \xi_X \subseteq [L^*, L^*] \).

 {\it Case 2}. The integer \( n \) is odd, \( n \geqslant 3 \).

Let \( X \subsetneq \{1, 2, \ldots, n\} \). Without loss of generality we assume that \( X \subseteq \{1, 2, \ldots, n-1\} \). If \( |X| < n-1 \), then \( A \xi_X \subseteq [L, L] \), by Case 1.

Let \( X = \{1, 2, \ldots, n-1\} \). Then for an arbitrary element \( a \in A \) we have
\[
[a \xi_1 \cdots \xi_{n-2} \xi_n, \xi_n \xi_{n-1}] = a \xi_1 \cdots \xi_{n-1}. \tag{18}
\]

Finally, let \( X = \{1, 2, \ldots, n\} \). For arbitrary elements \( a, b \in A_{\bar{0}} \cup A_{\bar{1}} \) we have
\[
[a \xi_1, b \xi_2 \cdots \xi_{n-1} \xi_n] = \pm \left( [a, b] + \frac{1}{2} (a \cdot d(b) - (n-1) d(a) b) \right) \xi_1 \cdots \xi_n. \tag{19}
\]

Choosing \( a = 1 \) we get
\[
-\frac{1}{2} d(b) \xi_1 \cdots \xi_n \in [L, L],
\]
hence \( d(A) \xi_1 \cdots \xi_n \in [L, L] \).

Now,
\[
[a, b] + \frac{1}{2} (a \cdot d(b) - (n-1) d(a) b) = [a, b] \textcolor{.}{+\frac{n}{2}} a \cdot d(b) \textcolor{.}{- \frac{n-1}{2}} d(ab),
\]
which implies that \( M_n \xi_1 \cdots \xi_n \subseteq [L, L] \) and completes the proof of the lemma.
\end{proof}

\begin{lemma}\label{lem:Kperf>2}
For \( n \geqslant 3 \) the Lie superalgebra \( [K(A:n), K(A:n)] \) is perfect.
\end{lemma}

\begin{proof}
Let \( n \geqslant  3 \). Denote \( L = K(A:n) \). By the formulas (17), (19) above the subspace \( M_n \xi_1 \cdots \xi_n \) lies in the sum \( \sum [A \xi_{X'}, A \xi_{X''}] \) where \( |X'|, |X''| < n \). This implies that
\[
M_n \xi_1 \cdots \xi_n \subseteq [[L, L], [L, L]].
\]

 Let \( X \subsetneq \{1, 2, \ldots, n\} \). By (16), (18) the only case that needs attention is when \( n = 2m+1 \) is odd, \( |X| = r < n \) with \(r\) odd. The inclusion (15) is applicable to  this case, giving \( A \xi_X \subseteq [A \xi_X, H] \), \( H = \sum_{i=1}^m F \zeta_i \eta_i \). If \( n \geqslant 3 \), then \( \zeta_i \eta_i \in [L, L] \) by Section {\ref{subsec:BasicConstrut-K}}. Hence,
\[
A \xi_X \subseteq [  [L, L], [L, L]],
\]
which completes the proof of the lemma.
\end{proof}

\begin{lemma}\label{lem:Kfingen>2}
Let \( A \) be a finitely generated associative commutative superalgebra with a contact bracket. Let \( L = K(A:n), n \geqslant 3 \). Then the Lie superalgebra \( [L, L] \) is finitely generated.

The Lie superalgebra \( L \) is finitely generated if and only if the subspace \( M_n \) has finite codimension in \( A \).
\end{lemma}

\begin{proof}
We showed in Lemma \ref{lem: KperfFirstLemma} (formulas (17), (19)) that the superalgebra \( [L, L] \) is generated by \( \sum_{X \subsetneq \{1, 2, \ldots, n\}} A \xi_X \). Let \( X \subsetneq \{1, 2, \ldots, n\} \),  \( i \notin X .\)

Let \( X = X' \cup X'' \). Since \( n \geqslant 3 \) we may assume that the subsets \( X' \cup \{i\} \) and \( X'' \cup \{i\} \) are proper. For arbitrary elements \( a, b \in A \) we have
\[
[a \xi_{X' \cup \{i\}}, b \xi_{X'' \cup \{i\}}] = \pm ab \xi_X.
\]

We showed that if the superalgebra \( A \) is generated by elements \( a_1, \ldots, a_m \in A_{\bar{0}} \cup A_{\bar{1}} \), then the Lie superalgebra \( [L, L] \) is generated by \( a_i \xi_X \), \( X \subsetneq \{1, 2, \ldots, n\} \).

If \( \dim_F A / M_n = \infty \) then \( \dim_F L / [L, L] = \infty \), which implies that the superalgebra \( L = K(A:n) \) is not finitely generated. On the other hand, if \( A = \sum_{i=1}^s F u_i + M_n \) then the superalgebra \( L \) is generated by the elements \( u_i \xi_1 \cdots \xi_n \), \( 1 \leqslant i \leqslant s \), and by generators of the superalgebra \( [L, L] \). This completes the proof of the lemma.
\end{proof}

\begin{lemma}\label{lem:K=2perf+fingen}
Let \( A \) be a finitely generated associative commutative superalgebra and let \( M_2 = A ,\) \((M_2 = \operatorname{Span}_F \left\{ [a, b] + a d(b) \middle| a, b \in A \right\}).\) Then the superalgebra \( K(A:2) \) is perfect and finitely generated.
\end{lemma}

\begin{proof}
The equality \( M_2 = A \) implies that the Lie superalgebra \( K(A:2) \) is generated by \( A \zeta_1 \) and \( A \eta_1 \). We have
\[
[a \zeta_1, b \zeta_1 \eta_1] = (-1)^{|b|+1} ab \zeta_1,
\]
\[
[a \eta_1, b \zeta_1 \eta_1] = (-1)^{|b|} ab \eta_1.
\]

This implies that if elements \( a_1, \ldots, a_m \in A_{\bar{0}} \cup A_{\bar{1}} \) generate \( A \) then the elements \( a_i \zeta_1, a_i \eta_1, a_i \zeta_1 \eta_1 \), \( 1 \leqslant i \leqslant m \), generate \( K(A:2) \). This completes the proof of the lemma.
\end{proof}

\subsection{Twisted Superalgebras \( K^{(2)}(A:n)\), \(n \geqslant 2 \)}\label{subsec:FinGen-twisted K} \mbox{}\\

\textcolor{.}{Let \(A = A^{(0)} \oplus A^{(1)}\) be a \( \mathbb{Z}/2\mathbb{Z} \)-graded contact superalgebra,}  \textcolor{.}{and let \(K^{(2)}(A:n) = A^{(0)} \otimes G(n-1) + A^{(1)} \otimes \xi_n G(n-1)
\) be the twisted type \(K\) Lie superalgebra.}  (See Section \ref{subsec:BasicConstrut-K twisted}) 

For an arbitrary integer \( r \), \( 0 \leqslant r \leqslant n \), and an arbitrary parity \( \sigma = 0 \) or \( 1 \) let
\[
M_r^{(\sigma)} = M_r \cap A^{(\sigma)}.
\] 

\textcolor{.}{
Define also 
\[
M_r^{(i,j)} = \operatorname{Span}_F \left\{ [a, b] + \frac{r}{2} a d(b) \middle| a \in A^{(i)},  b \in A^{(j)}\right\}.
\] It is clear that  \(M_r^{(i,j)} \subset M_r^{(i+j)}\). The relations 
\[
\bullet  \mbox{ }  M_r^{(0)}=M_r^{(0,0)} + M_r^{(1,1)} ,  \quad \bullet  M_r^{(0,1)} \subseteq   M_r^{(1,0)} = M_r^{(1)},  \quad  \bullet  \mbox{ }d(A^{(i+j)}) \subseteq M_r^{(i,j)}\]always hold, 
\textcolor{.}{unless when \(i=j=1\) or when \(i=1,j=0, r=2\). \footnote{Actually, \(M_r^{(0,1)} =   M_r^{(1,0)} \) holds for all \(r\neq 2\). }} 
}


 Reasoning as in \textcolor{.}{Section {\ref{subsec:FinGen-K(no.2)}},} we see that the superalgebra \( [K^{(2)}(A:n), K^{(2)}(A:n)] \) differs from the superalgebra \( K^{(2)}(A:n) \) \textcolor{.}{ in two components:} 

\begin{enumerate}
    \item The \( \xi_1 \cdots \xi_n \) component of the superalgebra \( [K^{(2)}(A:n), K^{(2)}(A:n)] \) is  \textcolor{.}{ \( M_n^{(1)} \xi_1 \cdots \xi_n \);
    \item The \( \xi_1 \cdots \xi_{n-1} \) component of \( [K^{(2)}(A:n), K^{(2)}(A:n)] \) equals \( (M_{n-1}^{(0,0)}+A^{(1)} \cdot A^{(1)}) \xi_1 \cdots \xi_{n-1} \)}.
\end{enumerate} 

Now we can prove Theorem \ref{theorem:A} for the  \textcolor{.}{class} \( K^{(2)}(A:n) \), \( n \geqslant 2 .\)

\begin{theorem}\label{thm:AforKtwisted}\mbox{}

(1) Let \( n \geqslant 3 \). Then the superalgebra \( [K^{(2)}(A:n), K^{(2)}(A:n)] \) is perfect. If the superalgebra \( A \) is finitely generated then so is the Lie superalgebra \( [K^{(2)}(A:n), K^{(2)}(A:n)] \).

(2) If \( A \) is a finitely generated associative commutative superalgebra  \textcolor{.}{such that \( M_2^{(1)} = A^{(1)} \) and \((M_{1}^{(0,0)}+A^{(1)} \cdot A^{(1)})=A^{(0)}\)} then the superalgebra \textcolor{.}{\( K^{(2)}(A:2) \)} is perfect. \textcolor{.}{In addition if \(A\) is finitely generated while \(A^{(0)}\) is a finite \(A^{(1)} \cdot A^{(1)}\)-module, then \( K^{(2)}(A:2) \) } is  finitely generated. 
\end{theorem}


\begin{remark}
\textcolor{.}{
   Note that \(A^{(0)} \cap A^{(1)}=0\): in particular \(1 \not\subset A^{(1)}\). This is a main reason why the similar proofs in   Section {\ref{subsec:FinGen-K(no.2)}} cannot be  applied verbatim to \( K^{(2)}(A:n) \).}
\end{remark}

\begin{proof}{\color{.}

Let \(L^{(0)}=A^{(0)} \otimes G(n-1)\), \(L^{(1)}=A^{(1)} \otimes \xi_n G(n-1)\), and \( L = K^{(2)}(A:n) \).  

 {\textit{Case $n \geqslant 3$:}} For any \(X \subset \{1,2,\ldots, n\}\) such that \(X\) is not \(\emptyset\) or \(\{n\}\),  one can always find \(j \in X\backslash\{n\}\) so that \([\xi_j, b\xi_X]=\pm b\xi_{X \backslash \{j\}}\). 

This shows  that \([L,L]\) differs from \(L\) at only two components (as was stated above). 

Now clearly \([a\xi_Y, b\xi_X]\in [[L,L],[L,L]]\) for all  \(X,Y \not\supset  \{1,2,\ldots, n-1\}\); the equalities (14) imply that these elements span almost all \(\xi\)-components of \([L,L]\) -- except for the two  aforementioned outliers.

So it remains to verify that the \( \xi_1 \cdots \xi_n \) and \( \xi_1 \cdots \xi_{n-1} \) components of \([L,L]\) and \([[L,L],[L,L]]\) are equal. 

Let \([n]=\{1,2,\ldots, n\}\), and \(a^{(i)}, b^{(j)}\) be arbitrary elements from \(A^{(i)}, A^{(j)}\). 

As the first component is spanned by elements of form  \([a^{(0)}\xi_1,b^{(1)}\xi_{[n]\backslash\{1\}}] \) while the second component is spanned by elements of form \([a^{(0)}\xi_1,b^{(0)}\xi_{[n]\backslash\{1,n\}}] \), and \([a^{(1)}\xi_n,b^{(1)}\xi_{[n]}] \), the verification is complete. 

This gives   \([L,L]=[[L,L],[L,L]]\) for \(n \geqslant 3\).

For finite generation, take any \(   X \subset [n]\) where \([n-1] \not\subset X\) and \(\emptyset,\{n\}\neq X\). 

Take any \(i\in [n-1]\backslash X\) and any \(j \in X\backslash\{n\}\). Making use of the ``going up'' equalities \([a \xi_{X' \cup \{i\}}, b \xi_{X'' \cup \{i\}}] = \pm ab \xi_X\) (from the proof of Lemma \ref{lem:Kfingen>2}) and ``going down''  equalities \([\xi_j, b\xi_X]=\pm b\xi_{X \backslash \{j\}}\) (from the current proof), one can obtain the desired result. 

In other words, let \(S=\{a_j\} \subset A^{(0)} \cup A^{(1)}\) be a finite set of generators of \(A\), then for all \(n \geqslant 3\), \( [L, L] \) is generated by the finite set  \( \{a_i \xi_X \mid a_i \in S, [n-1]\not\subset X \subsetneq [n]\} \cap L .\)

 {\textit{Case $n=2$:}} Again let \( L = K^{(2)}(A:n) \). Clearly \[L=(A^{(0)} \otimes 1)  \oplus (A^{(1)} \otimes \xi_2) \oplus (A^{(0)} \otimes \xi_1) \oplus (A^{(1)} \otimes \xi_1\xi_2).\]
 As before, \(L\) and \([L,L]\) differ only at the \(\xi_1\) and \(\xi_1\xi_2\) components (which we already computed). This immediately gives that \( M_2^{(1)} = A^{(1)} \) and \((M_{1}^{(0,0)}+A^{(1)} \cdot A^{(1)})=A^{(0)}\) implies \(L=[L,L].\)

 Now we prove finite generation:
 
 Under the assumption \( M_2^{(1)} = A^{(1)} \) the subspaces \( A^{(1)}\xi_2 \) and \( A^{(0)}\xi_1 \) generate \(L\).  Now assume that \(S= S^{(0)} \bigsqcup S^{(1)}\) is a finite set of generators of \(A\), where \(I,J,R\) are finite sets,  \(S^{(0)}=\{1, a_i\}_{i\in I} \subset A^{(0)}, {} S^{(1)}=\{b_j\}_{j\in J} \subset  A^{(1)}\), and \(T^{(0)}=\{c_r\}_{r\in R} \subset A^{(0)}\) being generators of \(A^{(0)}\) as an \(A^{(1)} \cdot A^{(1)}\)-module. 

To proceed, we list several structural commutator  identities:
\begin{align*}
&[a^{(0)} \xi_1, b^{(0)}] = \pm  ([a^{(0)},b^{(0)}]+\tfrac{1}{2}a^{(0)}d(b^{(0)}))\xi_1, \\
&[a^{(1)} \xi_1\xi_2, b^{(0)}] = \pm  ([a^{(1)},b^{(0)}]+a^{(1)}d(b^{(0)}))\xi_1\xi_2, \\
&[a^{(0)} \xi_1, b^{(0)}  \xi_1] = \pm  a^{(0)}b^{(0)}, \quad\quad\quad\quad [a^{(1)} \xi_2, b^{(1)} \xi_2] = \pm  a^{(1)}b^{(1)},\\
&[a^{(0)} \xi_1, b^{(1)}\xi_1\xi_2] = \pm  a^{(0)}b^{(1)}\xi_2,\quad\quad\ [a^{(1)} \xi_1\xi_2, b^{(1)} \xi_2] = \pm  a^{(1)}b^{(1)}\xi_1.
\end{align*}
It follows from the last two of these identities that the finite set of elements \( \xi_1, a_i \xi_1, c_r \xi_1, b_j \xi_2, b_j \xi_1\xi_2 \), \( i\in I, j\in J, r\in R   \), generate \( A^{(1)}\xi_2 \) and \( A^{(0)}\xi_1 \) and therefore \( K^{(2)}(A:2) \).}
\end{proof}






\subsection{Exceptional Cheng-Kac Superalgebras} \mbox{}\\

In \cite{martinez2003} it is shown that for any associative commutative superalgebra \( A \) with an even derivation \( d \), the Lie superalgebra \( CK(A,d) \) is perfect. If the superalgebra \( A \) is finitely generated then so is \( CK(A,d) \).

\section{Universal Central Extensions}

\subsection{Basic Definitions} \mbox{}\\

Let \( L \) be a perfect Lie superalgebra. An epimorphism \( \varphi : \widetilde{L} \to L \) is said to be \textit{central} if the superalgebra \( \widetilde{L} \) is perfect and the kernel \( \ker \varphi \) lies in the center of the superalgebra \( \widetilde{L} \).

There exists a unique universal central extension \( \widehat{L} \to L \) such that for any other central extension \( \widetilde{L} \to L \) there exists a homomorphism \( \widehat{L} \to \widetilde{L} \) that makes the appropriate diagram commutative (I. Schur\cite{schur1904}, H. Garland\cite{garland1980}).

\subsection{Constructing the Universal Central Extension} \mbox{}\\

Let \( h \in L_{\bar{0}} \) be an element such that all eigenvalues of \( \operatorname{ad}(h) \) are integers if $charF=0$ and belong to $\{0,1,\ldots, p-1\}$ if $charF=p$, and the superalgebra \( L \) decomposes into a sum of root spaces
\[
L = \sum_{i  } L_i,
\]
where \( L_i \) is the root space \textcolor{.}{(i.e. generalized eigenspace)}  that belongs to the eigenvalue \( i \).

\textcolor{.}{Furthermore} suppose that the superalgebra \( L \) is generated by \(
\sum\limits_{i \neq 0  } L_i.
\) This assumption immediately implies that \(L\) is perfect. Choose a finite generating system \( \{a_i\in L, i \in I \}\)  that consists of even or odd elements lying in \(\bigcup\limits_{k \neq 0  } L_k.\)
If \( a \in L_k \) then we say that \( \deg(a) = k \).

Consider the free superalgebra \( L(X) \) on the set of free generators \( X = \{x_i, i \in I\} \), where an element \( x_i \) has the same parity as \( a_i \). We define a \( \mathbb{Z} \)-grading on \( L(X) \) by assigning degrees \( \deg(x_i) = \deg(a_i) \). Consider also the epimorphism \( \pi : L(X) \rightarrow L \), \( x_i \mapsto a_i \), \( i \in I \). Let \( R = \ker \pi \), \( R = \sum\limits_{k \in \mathbb{Z}} R_k \). Let \( R^* = \sum\limits_{k \neq 0  } R_k \).

\begin{lemma}
\( \widehat{L} = \langle X \mid R^* = (0) \rangle \) is a universal central extension of \( L \).
\end{lemma}

\begin{proof}
For an element \( f \in L(X) \) let \( \hat{f} \) denote its image in \( \widehat{L} \). In particular, let \( \hat{x}_i \) be the image of a generator \( x_i \). 

There exists an element \( f \in L(X)_{\bar{0}} \) such that \( h = f(a_1, a_2, \ldots) \); for each \( i \in I \), there exists also an integer \( n_i \geqslant 1 \) such that
\[
(\operatorname{ad}(f(a_1, a_2, \ldots)) - \deg(a_i))^{n_i} a_i = 0.
\]
Hence
\[
(\operatorname{ad}(f({x}_1, {x}_2, \ldots)) - \deg(a_i))^{n_i} \cdot {x}_i \in R_{\deg(a_i)} \subseteq R^*.
\]
It implies that
\(
\hat{x}_i \in [\widehat{L}, \widehat{L}],
\) so the superalgebra \( \widehat{L} \) is perfect.

Let us show that 
\[
\operatorname{id}_{L(X)}(R^*) \supseteq [R, L(X)], \tag{20}
\] where \(\operatorname{id}_{L(X)}(R^*)\) is the ideal of \( L(X) \) generated by \(R^*\).

Indeed, \( [R_0, \hat{x}_i] \subseteq R_{\deg(\hat{x}_i)} \subseteq R^* \), which implies the inclusion (20). From (20) it follows that the kernel of the natural epimorphism \( \pi : \widehat{L} \rightarrow L \), \( \hat{x}_i \mapsto a_i \), \( i \in I \), lies in the center of \( \widehat{L} \). Therefore we have shown that \( \widehat{L} \to L \) is a central extension.

Now let \( L' \) be a perfect Lie superalgebra. Let \( \sigma : L' \to L \) be an epimorphism and let \( \ker \sigma \subseteq Z(L') \). Let \( h' \in L' \) be a preimage of the element \( h \) under the epimorphism \( \sigma \), \( \sigma(h') = h \).

For an arbitrary element \( a \in L' \) there exists \( n \geqslant 1 \) such that
\[
\prod_{i=-n}^n (\operatorname{ad}(h) - i)^{n} \sigma(a) = 0.
\]
Hence,
\[
\prod_{i=-n}^n (\operatorname{ad}(h') - i)^{n}  a \in Z(L'),
\]
hence
\[
\operatorname{ad}(h')   \prod_{i=-n}^n (\operatorname{ad}(h') - i)^{n}   a = 0.
\]
This implies that
\[
L' = \sum_{i \in \mathbb{Z}} L'_i,
\]
where \( L'_i \) is the root space of \( \operatorname{ad}(h') \) that corresponds to the eigenvalue \( i \).

Let \( a'_i \in L'_{\deg(a_i)} \) be a preimage of the element \( a_i \). Since elements \( a'_i \), \( i \in I \), generate \( L' \) modulo the center and the algebra \( L' \) is perfect, it follows that the elements \( a'_i \), \( i \in I \), generate \( L' \).

Let \( r(x_1, x_2, \ldots) \in R_i \), \( i \neq 0 \). We have \( r(a_1, a_2, \ldots) = 0 \), hence \( r(a'_1, a'_2, \ldots) \in Z(L') \). We have
\[
(\operatorname{ad}(h') - i)^n r(a'_1, a'_2, \ldots) = 0
\]
for some \( n \), which implies \( r(a'_1, a'_2, \ldots) = 0 \). Hence \( \hat{x}_i \mapsto a'_i \), \( i \in I \), extends to an epimorphism \( \widehat{L} \to L' \). This completes the proof of the lemma.
\end{proof}

\section{Sufficient Conditions for Finite Presentability of Universal Central Extensions}\label{sec:CondFinPresUCE} 


In this chapter we formulate some abstract conditions for finite presentability of universal central extensions. In the next chapter we will verify these conditions for superconformal algebras.

Suppose that a Lie superalgebra \( L \) contains an element \( h \in L_{\bar{0}} \) such that all eigenvalues of \( \operatorname{ad}(h) \) are integers and \( L \) decomposes into a sum of root spaces
\[
L = \sum_{i=-m}^{m} L_i,
\]
and the characteristic of the field \( F \) is zero or greater than \(2m\).

We assume that the superalgebra \( L \) is generated by \( \sum_{i \neq 0} L_i \). In other words,
\[
L_0 = \sum_{i=1}^{m} [L_{-i}, L_i].
\]

Suppose that \( (L_{\bar{0}})_0 \) contains a finite dimensional abelian subalgebra \( V \) with  \( h \in V \). The subspace \( V \) lies in a bigger   finite  dimensional abelian subalgebra \( \widetilde{V} \subset   (L_{\bar{0}})_0   \).


\textbf{Main Assumptions.}

We make the following assumptions:

(1) There exists a bilinear mapping \( f: V \times V \to \widetilde{V} \). For each \( i \neq 0 \), let \( S_i \) be the \( F \)-span of all operators
\[
\operatorname{ad}(v') \operatorname{ad}(v'') - i \operatorname{ad}(f(v', v'')) \in \operatorname{End}_F(L_i), \quad\quad \mbox{ where } v', v'' \in V.
\]
We assume that for each  \( i \neq 0 \) 
\[
S_i^2 = (0).
\]

(2) Let \( \mathcal{A} \) denote the associative commutative subalgebra of \( \operatorname{End}_F(L) \) generated by the identity operator and by all operators \( \operatorname{ad}(v): L \to L \), \( v \in V \). We assume that all root spaces \( L_i \) are finitely generated \( \mathcal{A} \)-modules. More precisely, there exist finite dimensional graded subspaces \( X_i \subseteq L_i \), \( i \neq 0 \), such that \( L_i = \mathcal{A} X_i \).

(3) We assume that
\[
[V, [V, L_0]] = (0). \tag{21}
\]

It implies that for arbitrary elements \( x \in L_{-i} \), \( y \in L_i \), \( i \neq 0 \), we have 
\[
[\operatorname{ad}(V)^2 x, y] \subseteq [\operatorname{ad}(V)x, \operatorname{ad}(V)y] + [x, \operatorname{ad}(V)^2 y].
\]

Our aim in this chapter is to prove the following proposition:

\begin{proposition}\label{prop: 4.1}
The universal central extension \( \widehat{L} \) \textcolor{.}{(of \(L\) satisfying the above assumptions) } is finitely presented.
\end{proposition}


\noindent\textbf{Specifying a Finite System of Relations (of \(\widehat{L}\)).}
\mbox{}

We will consider several finite families of relations that hold in the superalgebra \( L \). There exist finite dimensional graded subspaces \( Y_i \subseteq L_i \), \( i \neq 0 \), such that
\[
\widetilde{V} \subseteq \sum_{i=1}^{m} [Y_{-i}, Y_i].
\]
This implies that the subspace \( \sum_{i \neq 0} X_i + \sum_{i \neq 0} Y_i \) generates \( L \), where \(X_i\) are the subspaces   defined in part (2) of the main assumptions. We will view all relations as relations in the generators \( \sum_{i \neq 0} X_i + \sum_{i \neq 0} Y_i \).
 \footnote{
Let \( \{e_p\}_p, \{f_q\}_q \) be bases of the subspaces $V, \widetilde{V}$ (in \(L\)) respectively. We \textbf{fix} expressions of $e_p, f_q$  in basis elements of $X_i$'s, $Y_i$'s: \(e_p=e_p(x,y), f_q=f_q(x,y).\) \textbf{Caution:} appearances of \({V} \mbox{ (resp. } \widetilde{V})\) in the relations are intended only as \textbf{shorthand} for linear combinations of  $e_p(x,y)$ (resp. $f_q(x,y)$).}

Consider the following relations:


We consider the actions of \(\widetilde{V}\) and \( S_i\) on the generators \footnote{Following the previous footnote, relation (R1) mean
\[[[f_{q_1}(x,y), f_{q_2}(x,y)], X_{\pm i}]=[[f_{q_1}(x,y), f_{q_2}(x,y)], Y_{\pm i}]=0.\]
In the same way we understand all other relations.}

\begin{align*}
[[\widetilde{V}, \widetilde{V}], X_{\pm i}] &= [[\widetilde{V}, \widetilde{V}], Y_{\pm i}] = (0), \quad 1 \leqslant i \leqslant m, \tag{R1} \\
S_i^2 X_i &= (0), \quad i \neq 0, \tag{R2}
\end{align*}

\textit{\textbf{For all \( -m \leqslant i,j \leqslant m \) such that \( i \neq 0, j \neq 0, i+j \neq 0 \)}}, we consider the inclusions
\[
[\operatorname{ad}(V)^p S_i^k X_i, \operatorname{ad}(V)^q S_j^l X_j] \subseteq \mathcal{A} X_{i+j}, \tag{R3}
\]
where \( p+q \leqslant 1 \); \( 0 \leqslant k, l \leqslant \textcolor{.}{1}\). If \( i+j \notin [-m, m] \), then it is assumed that the right hand side is \( (0) \).

We also have
\[
[[\operatorname{ad}(V)^p S_i^k X_i, \operatorname{ad}(V)^q S_{-i}^l X_{-i}], \operatorname{ad}(V)^r S_i^t X_i] \subseteq \mathcal{A} X_i \tag{R4}
\]
for any \( -m \leqslant i \leqslant m \), \( i \neq 0 \), \( p+q+r \leqslant 4 \); \( k, l, t =0 \mbox{ or }1 \).

\[
Y_i \subseteq \mathcal{A} X_i, \quad i \neq 0, \tag{R5}
\]
and
\[
[\widetilde{V}, X_i] \subseteq \mathcal{A} X_i, \quad i \neq 0, \tag{R6}
\]
Inclusions (R3 - R6)  can be expressed as a finite system of relations.

\noindent{}(R7) is the set of relations that present the finitely presented \footnote{\textcolor{.}{\(\mathcal{A}\) is Noetherian.}} \( \mathcal{A} \)-module \( L_i \).

\noindent{}(R8) is the finite system of relations that reflect the equalities \[ [V, [V, [X_i, X_{-i}]]] = 0 , \  i \neq 0 .\]

Let \( \widetilde{L} \) be the Lie superalgebra presented by generators \( X_i, Y_i \), \( i \neq 0 \), and by the relations (R1)-(R8).

We identify the subspaces \( V \subset \widetilde{V} \) with subspaces in \( \widetilde{L} \) in the natural way, \( V \subset \widetilde{V} \subset \widetilde{L}_{\bar{0}} \). Abusing notation we denote the subalgebra of \( \operatorname{End}_F(\widetilde{L}) \) generated by \( \operatorname{Id} \) and the adjoint operators \( \operatorname{ad}(v): \widetilde{L} \to \widetilde{L} \), \( v \in V \), as \( \mathcal{A} \).

\begin{remark}
 \textcolor{.}{One should use caution when comparing \(V, \widetilde{V} \subset L \) and their counterparts \(V, \widetilde{V} \subset \widetilde{L}:\) The former are \textup{abelian subalgebras} of \(L\), while the latter are \textup{only plain subspaces} of a finitely presented  Lie algebra \(\widetilde{L}\) that are \textbf{not}  guaranteed to be abelian.} 
\end{remark}



\noindent\textbf{Technical Lemmas.}
\mbox{}

\begin{lemma}\label{lem:4.1}
Let \( F[T_1, \ldots, T_k] \) be the algebra of polynomials \textcolor{.}{in \( k \) variables}.
\begin{enumerate}
\item[(1)] Let \( \gamma \in F \), \( \gamma \neq -1 \). The element \( T_1^2 \) lies in the ideal of \( F[T_1, T_2] \) generated by \( T_1 + T_2 \) and \( T_1^2 + \gamma{T_2^2} \). More precisely,
\[
T_1^2 = \frac{\gamma}{1+\gamma} (T_1 + T_2)(T_1 - T_2) + \frac{1}{1+\gamma} (T_1^2 + \gamma{T_2^2}).
\]

\item[(2)] \( T_1 T_3 \) lies in the ideal of \( F[T_1, T_2, T_3] \) generated by \( T_1 + T_2 + T_3 \), \( T_1^2 - T_2^2 + T_3^2 \). More precisely,
\[
T_1 T_3 = \frac{1}{2} ((T_1 + T_2 + T_3)(T_1 - T_2 + T_3) - (T_1^2 - T_2^2 + T_3^2)).
\]
\end{enumerate}
\end{lemma}

\begin{proof}
Both computations are straightforward.
\end{proof}


In what follows \(\widetilde{L}\) denotes the finitely presented algebra defined above and the same notation will be used.

\begin{lemma}\label{lem:4.2}  
The following identities are satisfied in \(\widetilde{L}\):

\begin{enumerate}
\item[(1)] Let \( i,j \in [-m, m] \setminus \{0\} \), \( i+j \neq 0 \). Let \( x \in \mathcal{A} X_i \), \( y \in \mathcal{A} X_j \), \( v \in V \), \( w = f(v, v) \). Then
\[
[[v, [v, x]], y] = \frac{i}{i+j} [v, [v, [x, y]] - {\color{.}{2}}[x, [v, y]]] + \frac{\color{.}j}{i+j} [( \operatorname{ad}(v)^2 - i \operatorname{ad}(w))x, y] 
\]
\[
+ \frac{\color{.}i}{i+j} [x, ( \operatorname{ad}(v)^2 - j \operatorname{ad}(w))y] + \frac{ij}{i+j} [w, [x, y]].
\]

\item[(2)] Let \( i \in [-m, m] \setminus \{0\} \); \( x, z \in \mathcal{A} X_i \), \( y \in \mathcal{A} X_{-i} \), \( v \in V \), \( w = f(v, v) \). Then
\[
\textcolor{.}{2}[[[v, x], y], [v, z]] =  [v, [[[v, x], y], z] - [[x, [v, y]], z] + [[x, y], [v, z]]]
\]
\[
- [[( \operatorname{ad}(v)^2 - i \operatorname{ad}(w))x, y], z] + [[x, ( \operatorname{ad}(v)^2 + i \operatorname{ad}(w))y], z] 
\]
\[
- [[x, y], ( \operatorname{ad}(v)^2 - i \operatorname{ad}(w))z] \textcolor{.}{-} i[w, [[x, y], z]].\]
\end{enumerate}
\end{lemma}

\begin{proof}
(1) Consider the tensor product \(\mathcal{A}X_i \otimes  \mathcal{A}X_j \) and the commuting operators
\[
T_1: x \otimes y \mapsto [v, x] \otimes y, \quad T_2: x \otimes y \mapsto x \otimes [v, y].
\]
Let \( \gamma = i/j \). Assertion (1)  follows from Lemma  {\ref{lem:4.1}(1)}, and the identity {\color{.} {\[[v, [v, [x, y]] - {2}[x, [v, y]]]=[\operatorname{ad}(v)^{2}x,y]-[x, \operatorname{ad}(v)^{2}y].\]}}

(2) Consider \( \mathcal{A}X_i \otimes \mathcal{A}X_{-i} \otimes \mathcal{A}X_i \), and the commuting operators
\[
T_1: x \otimes y \otimes z \mapsto [v, x] \otimes y \otimes z, \ \
T_2: x \otimes y \otimes z \mapsto x \otimes [v, y] \otimes z, \ \ 
T_3: x \otimes y \otimes z \mapsto x \otimes y \otimes [v, z].
\]
Similar to (1), assertion (2)  follows from Lemma {\ref{lem:4.1}(2)}. 

This completes the proof of the lemma.
\end{proof}


\begin{lemma}\label{lem:4.3}
 For any \( -m \leqslant i, j \leqslant m \), such that \( i \neq 0, j \neq 0, i+j \neq 0 \), we have in $\widetilde{L}$ the inclusion:
\[
[\mathcal{A} X_i, \mathcal{A} X_j] \subseteq \mathcal{A} X_{i+j}.
\]
If \( i+j \notin [-m, m] \) then the right hand side is \( (0) \).
\end{lemma}






\begin{proof}

We will use induction on \(p + q\). If \(p + q \leqslant 1\), then the inclusion follows from (R3).

Let \(p + q \geqslant 1\). By the Jacobi identity and the induction assumption we can assume that \(q = 0\) and therefore \(p \geqslant 1\). So, one considers
\[
[\operatorname{ad}(v_1)\operatorname{ad}(v_2)\cdots \operatorname{ad}(v_p)S_i^k X_i, S_j^k X_j].
\]
Since the elements \(v_1, v_2\) commute, while \(0 \neq 2\) in \(F\), without loss of generality we can assume that \(v_1 = v_2 = v\).

Let \(x, y\) be elements from \(\operatorname{ad}(v_3) \cdots \operatorname{ad}(v_p) S_i^k X_i\) and \(S_j^l X_j\) respectively, \(x \in \mathcal{A} X_i\), \(y \in \mathcal{A} X_j\).

Let us apply Lemma \ref{lem:4.2} (1) to \([[v, [v, x]],y]\) and consider all summands on the right hand side of Lemma \ref{lem:4.2} (1) separately.

The element \([v,[ x, y]] - 2 [x, [v, y]]\) lies in \(\mathcal{A} X_{i+j}\) since \(p + q\) went down by 1. The element \[[(\operatorname{ad}(v)^2 - i\operatorname{ad}(w)) x, y] \in [\operatorname{ad}(v_3) \cdots \operatorname{ad}(v_p) S_i^{k+1} X_i; S_j^l X_j]\] lies in \(\mathcal{A} X_{i+j}\) since \(p + q\) went down by 2. Similarly, \[[x, (\operatorname{ad}(v)^2 - j\operatorname{ad}(w))y] \in [\operatorname{ad}(v_3) \cdots \operatorname{ad}(v_p) S_i^k X_i; S_j^{l+1} X_j] \subseteq \mathcal{A} X_{i+j}\] since \(p + q\) went down by 2. 

Finally, \([x, y] \in \mathcal{A} X_{i+j}\) since \(p + q\) went down by 2 and, therefore, \[[w, [x, y]] \in \mathcal{A} X_{i+j}. \qedhere \]
\end{proof}


\begin{lemma}\label{lem: 4.4}
  For arbitrary indices \( -m \leqslant i,j \leqslant m \), \( i \neq 0, j \neq 0 \), we have in $\widetilde{L}$ the inclusion:
\[
[[\mathcal{A} X_i, \mathcal{A} X_{-i}], \mathcal{A} X_j] \subseteq \mathcal{A} X_j.
\]
\end{lemma}

\begin{proof}
If \( j \neq \pm i \), then the assertion follows from Lemma \ref{lem:4.3}. 

Let \( j = i \). We will use induction on \[p+q+r\]
to show that
\[
[[\operatorname{ad}(V)^p S_i^k X_i, \operatorname{ad}(V)^q S_{-i}^l X_{-i}], \operatorname{ad}(V)^r S_i^t X_i] \subseteq \mathcal{A} X_i.
\]

If \( k + l + t \geqslant 4 \), then the expression is equal to zero.

Note that if  \( p+q+r \leqslant 4 \), then the inclusion follows from (R4). Thus, suppose that \( p+q+r \geqslant 5 \). Let \( x, z \in X_i, y \in X_{-i} \). Consider an element \(\omega\) in 
\[
 [[\operatorname{ad}(a_1)\cdots \operatorname{ad}(a_p) S_i^k x, \operatorname{ad}(b_1)\cdots \operatorname{ad}(b_q) S_{-i}^l y], \operatorname{ad}(c_1)\cdots \operatorname{ad}(c_r) S_i^t z],
\]
where \( a_{\mu}, b_{\nu}, c_{\kappa} \in V \).

\begin{enumerate}
\item[1)] Suppose that \( a_1 = c_1 \). Then the assertion follows from Lemma {\ref{lem:4.2}}(2) and the induction assumption.

\item[2)] Suppose that \( p \geqslant 2, r \geqslant 1 \). Since the elements \( a_1, a_2 \) commute, we assume that \( a_1 = a_2 \). Then
\begin{align*}
\omega &= [[\operatorname{ad}(a_1)\operatorname{ad}(a_1)\cdots x, \cdots y], \operatorname{ad}(c_1)\cdots z]\\ 
&= [[\operatorname{ad}(a_1 + c_1)\operatorname{ad}(a_1)\cdots x, \cdots y], \operatorname{ad}(a_1 + c_1)\cdots z]\\
& \quad - [[\operatorname{ad}(a_1)\operatorname{ad}(a_1)\cdots x, \cdots y], \operatorname{ad}(a_1)\cdots z]\\
& \quad - [[\operatorname{ad}(c_1)\operatorname{ad}(a_1)\cdots x, \cdots y], \operatorname{ad}(c_1)\cdots z]\\
& \quad - [[\operatorname{ad}(c_1)\operatorname{ad}(a_1)\cdots x, \cdots y], \operatorname{ad}(a_1)\cdots z].
\end{align*}
 

All the summands on the right hand side belong to the previously considered case 1.

\item[3)] Suppose that \( q \geqslant 2 \).
\textcolor{.}{For $\alpha, \beta, \gamma \geqslant 0$, \(
\alpha + \beta + \gamma = q - 2, 
\) we have the following inclusion: }\[
[[\operatorname{ad}(V)^p x, \operatorname{ad}(V)^q y], \operatorname{ad}(V)^r z] \subseteq 
\sum\limits_{\textcolor{.}{\substack{\alpha+\beta+\gamma\\=q-2}}} \operatorname{ad}(V)^\alpha [[\operatorname{ad}(V)^{p+\beta} x, \operatorname{ad}(V)^2 y], \operatorname{ad}(V)^{r+\gamma} z].
\]

If \( \alpha \geqslant 1 \), then the summands are covered by the induction assumption. 

\textcolor{.}{Therefore,} let \( \alpha = 0 \). Since \( p + q + r \geqslant 5 \) it follows that \( (p+\beta) + (r+\gamma) \geqslant 3 \).


Now,
\[
[[\operatorname{ad}(V)^p x, \operatorname{ad}(V)^q y], \operatorname{ad}(V)^r z] =
\]
\[
[\operatorname{ad}(V)^p x, [\operatorname{ad}(V)^q y, \operatorname{ad}(V)^r z]] \mod \mathcal{A} X_i
\]
Indeed, \( [\mathcal{A} X_i, \mathcal{A} X_i] \subseteq \mathcal{A} X_{2i} \) and \( [\mathcal{A} X_{2i}, \mathcal{A} X_{-i}] \subseteq \mathcal{A} X_{i} \) by Lemma \ref{lem:4.3}.

If both \( p+\beta \) and \( r+\gamma \) are nonzero, then without loss of generality we assume \(p+\beta \geqslant 2, r+\gamma \geqslant 1.\) Hence the summand is covered by case (2).

Let \( p+\beta = 0, r+\gamma \geqslant 3 \); the case \( p+\beta \geqslant 3, r+\gamma = 0 \) is similar. By (R8) we have
\[
[x, \operatorname{ad}(V)^2 y] \subseteq \sum\limits_{\color{.} s \geqslant 1} [\operatorname{ad}(V)^{\color{.}s} x, \operatorname{ad}(V)^{\color{.}2-s} y],
\]

The summand \( [[\operatorname{ad}(V)^{\color{.}s} x, \operatorname{ad}(V)^{\color{.}2-s} y], \operatorname{ad}(V)^{r+\gamma} z] \) with \( \color{.} s =1,2 \) is covered by case 2). This completes the proof in the case \( q \geqslant 2 \).

\item[4)] Now let \( q \leqslant 1 \) and therefore \( p + r \geqslant 4 \). If both \( p \) and \( r \) are nonzero, then the summand is covered by case 2).

Let \( r = 0, p \geqslant 4 \). Then
\[
[[\operatorname{ad}(V)^p x, \operatorname{ad}(V)^q y],  z] \subseteq
\]
\[
\sum\limits_{\textcolor{.}{\substack{\alpha+\beta+\gamma\\=p-2}}} \operatorname{ad}(V)^\alpha [[\operatorname{ad}(V)^2 x, \operatorname{ad}(V)^{q+\beta} y], \operatorname{ad}(V)^{\gamma} z].
\]

If \( \alpha \geqslant 1 \) then the induction assumption applies.

Let \( \alpha = 0 \). If \( \gamma \geqslant 1 \) then the summand is covered by case 2). Let \( \gamma = 0 \), then 
\[
\beta \geqslant 2. 
\]
 \text{Hence} \(q + \beta \geqslant 2\), which is a case considered in 3). This completes the proof of the lemma. \qedhere
\end{enumerate}
\end{proof}

We are now in place to prove Proposition \ref{prop: 4.1}:


\begin{proposition*}
The universal central extension \( \widehat{L} \) is finitely presented.
\end{proposition*}

\begin{proof}
Lemmas \ref{lem:4.3} and \ref{lem: 4.4} imply that
\[
\widetilde{L} = \sum_{i \neq 0} \mathcal{A} X_i \oplus \sum_{i \neq 0} [\mathcal{A} X_i, \mathcal{A} X_{-i}].
\]

The identical mapping \( X_i \rightarrow X_i \) extends to an epimorphism \( \varphi: \widetilde{L} \rightarrow \widehat{L} \), where \( \widehat{L} \) is the universal central extension of \( L \). The relations (R7) imply that \( \ker \varphi \subseteq \sum\limits_{i \neq 0} [\mathcal{A} X_{-i}, \mathcal{A} X_i] \). Hence, \( \ker \varphi \) lies in the center of \( \widetilde{L} \). Therefore we conclude that \( \widetilde{L} \cong \widehat{L} \). This completes the proof of the proposition.
\end{proof}


\section{Superconformal Algebras}

In this chapter we will check that the assumptions of Chapter  \ref{sec:CondFinPresUCE} are satisfied for superconformal algebras considered above. It will imply   Theorem \ref{theorem:B}. 

\textcolor{.}{In what follows, let \( U \) be a \textit{finite dimensional} graded generating subspace of \( A \), with \( 1 \in U_{\bar{0}} \). \footnote{{\textcolor{.}{We remind the reader that if $A$ is a finitely generated associative commutative superalgebra, then $A_{\bar{1}}$ is a finitely generated $A_{\bar{0}}$-module.}}}}

\subsection{Superalgebras \( W(A:n) \), \( n \geqslant 1 \) and \( S(\delta,n) \), \( n \geqslant 2 \)}\label{subsec:5.1}
\mbox{}\\


Let \( n \geqslant 2 \), \( L = W(A:n) \), \( h = \xi_1 \frac{\partial}{\partial \xi_1} - \xi_2 \frac{\partial}{\partial \xi_2} \). The operator \( \operatorname{ad}(h) \) has eigenvalues \( -2, -1, 0, 1, 2 \), and
\[
L = L_{-2} + L_{-1} + L_0 + L_1 + L_2
\]
is a direct sum of eigenspaces. We have
\begin{align*}
L_{-2} &= \sum A \cdot \xi_{i_1} \cdots \xi_{i_r} \xi_2 \frac{\partial}{\partial \xi_1}, \\
L_{-1} &= \sum \xi_{i_1} \cdots \xi_{i_r} \left( \xi_2 \operatorname{Der} A + A \frac{\partial}{\partial \xi_1} + \sum_{k=1}^n A \xi_2 \xi_k \frac{\partial}{\partial \xi_k}  \textcolor{.}{+\sum_{j=3}^n A \xi_2 \frac{\partial}{\partial \xi_j}} \right), \\
L_1 &= \sum \xi_{i_1} \cdots \xi_{i_r} \left( \xi_1 \operatorname{Der} A + A \frac{\partial}{\partial \xi_2} + \sum_{k=1}^n A \xi_1 \xi_k \frac{\partial}{\partial \xi_k} \textcolor{.}{+\sum_{j=3}^n A \xi_1 \frac{\partial}{\partial \xi_j}} \right), \\
L_2 &= \sum A \cdot \xi_{i_1} \cdots \xi_{i_r} \xi_1 \frac{\partial}{\partial \xi_2},
\end{align*}
where \( 3 \leqslant i_1 < \cdots < i_r \leqslant n,\)  
 \textcolor{.}{ \(\xi_{\emptyset} = 1\) when  \(r=0\)}.


Let \( z = \xi_{i_1} \cdots \xi_{i_r} \), \( 3 \leqslant i_1 < \ldots < i_r \leqslant n \), \( d \in \operatorname{Der} A \), \( a \in A_{\bar{0}} \). We have
\begin{align*}
[a h, z \xi_2 d] &= -a z \xi_2 d - z \xi_2 d(a) \xi_1 \frac{\partial}{\partial \xi_1}, \tag{22} \\
[a h, z \xi_1 d] &= a z \xi_1 d + z \xi_1 d(a) \xi_2 \frac{\partial}{\partial \xi_2}. \tag{23}
\end{align*}

For all elements \( u \in L_i \), \( i \neq 0 \), not involving summands from
\[
\xi_{i_1} \cdots \xi_{i_r} \xi_2 \operatorname{Der} A, \quad \xi_{i_1} \cdots \xi_{i_r} \xi_1 \operatorname{Der} A,
\]
we have
\[
[a h, u] = i a u. \tag{24}
\]

Let \( V = U_{\bar{0}} h \), \( \widetilde{V} = (U_{\bar{0}})^2 h \). For elements \( v' = u' h \), \( v'' = u'' h \); \( u', u'' \in U_{\bar{0}} \), define \( f(v', v'') = u' u'' h \).

The equalities (22), (23), (24) imply
\[
S_i^2 L_i = (0), \quad i = \pm 1, \pm 2.
\]

The same formulas imply that each eigenspace \( L_i \), \( i \neq 0 \), is a finitely generated \( \mathcal{A} \)-module.

Let us check the condition (21). We have
\[
L_0 = \sum \xi_{i_1} \cdots \xi_{i_r} \left( \operatorname{Der} A + \sum_{k=1}^n A \xi_k \frac{\partial}{\partial \xi_k} \right),
\]
where in each summand the set \( \{i_1, \ldots, i_r\} \) either does not include 1 and 2 or includes both 1 and 2.

Denote \( L_0' = \sum \xi_{i_1} \cdots \xi_{i_r} \left( \sum\limits_{k=1}^n A \xi_k \frac{\partial}{\partial \xi_k} \right) \), we have
\[
[A_{\bar{0}} (\xi_1 \frac{\partial}{\partial \xi_1}- \xi_2 \frac{\partial}{\partial \xi_2}), L_0] \subseteq L_0',
\]
\[
[A_{\bar{0}} (\xi_1 \frac{\partial}{\partial \xi_1}- \xi_2 \frac{\partial}{\partial \xi_2}), L_0'] = (0).
\]
This implies (21), so we showed that the superalgebra \( W(A:n) \), \( n \geqslant 2 \) satisfies all the assumptions of Chapter {\ref{sec:CondFinPresUCE}}.


Now let \( \delta : \operatorname{Der} A \to A \) be a divergence mapping, let \( L = S(\delta, n) \), \( n \geqslant 2 \).

We notice that
\[
h = \xi_1 \frac{\partial}{\partial \xi_1} - \xi_2 \frac{\partial}{\partial \xi_2} \in S(\delta, n) \subset W(A:n),
\]
\[
V \subseteq \widetilde{V} \subseteq S(\delta, n).
\]

The decomposition \( L = L_{-2} + L_{-1} + L_0 + L_1 + L_2 \) of the superalgebra \( L \) into eigenspaces is the restriction of the decomposition of \( W(A:n) \). It implies that all \( L_i \), \( i \neq 0 \), are finitely generated \( \mathcal{A} \)-modules and \( S_i^2 L_i = (0) \), \( i \neq 0 \).

The assumption (21) also follows, since it holds for the algebra \( W(A:n) \).

\subsection{The Case \( W(A:1) \) with \( A = A^{\mathcal{D}} \)} \mbox{}\\ 

Let \( L = W(A:1) \) and assume that \( A = A^{\mathcal{D}} \). We have only one Grassmann variable \( \xi \).

Let \( h = \xi \frac{\partial}{\partial \xi} \), \( L = L_{-1} + L_0 + L_1 \), \( L_{-1} = A \frac{\partial}{\partial \xi} \), \( L_1 = \xi \operatorname{Der} A \).

For arbitrary elements \( a \in A_{\bar{0}} \), \( x \in A\partial/\partial{\xi} \)  , \( d \in \operatorname{Der} A \)  we have
\begin{align*}
[a \xi \frac{\partial}{\partial \xi}, x  ] &= -a{x}, \\
[a \xi \frac{\partial}{\partial \xi}, \xi d] &= a \xi d.
\end{align*}

\textcolor{.}{Define 
 \( V , \widetilde{V}, \mbox{ and } f: V \times V \to \widetilde{V} \) as in the previous case. It follows that}  \( L_i \), \( i = \pm 1 \) are finitely generated \( \mathcal{A} \)-modules and \( S_i L_i = (0) \), \( i = \pm 1 \).

We have \( L_0 = \operatorname{Der} A + A \xi \frac{\partial}{\partial \xi} \). Furthermore,
\[
[A_{\bar{0}} \xi \frac{\partial}{\partial \xi}, L_0] \subseteq A \xi \frac{\partial}{\partial \xi}, \quad [A_{\bar{0}} \xi \frac{\partial}{\partial \xi}, A \xi \frac{\partial}{\partial \xi}] = (0).
\]

This implies the assumption (21).

\subsection{Cheng-Kac Superalgebra \( CK(A,d) \)} \mbox{}\\

Let \( W = \sum_{i \geqslant 0} A d^i \) be the Weyl algebra. Let \( P = A + A d \). From the construction in Section {\ref{subsec:CK construction}} it follows that \( CK(A,d) \subseteq M_8(W) \) and, moreover, \( CK(A,d) \subseteq M_8(P) \). Let \( h = \operatorname{diag}(1, -1, 0, 0, -1, 1, 0, 0) \).

For an element \( a \in A_{\bar{0}} \) denote
\[
h(a) = \operatorname{diag}(a, -a, 0, 0, -a, a, 0, 0).
\]
Let
\[
V = h(U_{\bar{0}}), \quad \widetilde{V} = h(U_{\bar{0}}^2).
\]
For elements \( v' = h(u') \), \( v'' = h(u'') \), \( u', u'' \in U_{\bar{0}} \), we let \( f(v', v'') = h(u'u'') \).

As above, for an element \( w \in W \) and integers \( 1 \leqslant i, j \leqslant 8 \), the matrix \( e_{ij}(w) \in M_{8}(W) \) has the element \( w \) at the intersection of the \( i \)-th row and the \( j \)-th column and zeros everywhere else.

The element \( e_{ij}(w) \) belongs to one of the eigenvalues $ 0, \pm 1, \pm 2$


If \( e_{ij}(w) \) belongs to the eigenvalue \( k \neq 0 \) then
\[
S_k e_{ij}(P) \subseteq e_{ij}(A), \quad S_k e_{ij}(A) = (0).
\]
Hence,
\[
S_k^2 e_{ij}(P) = (0).
\] \textcolor{.}{From here} it is \textcolor{.}{a} straightforward \textcolor{.}{computation} that for \( k \neq 0 \), the space \( e_{ij}(P) \) is a finitely generated  \( \mathcal{A} \)-module.

The superalgebra \( L = CK(A,d) \) also decomposes as \( L = L_{-2} + L_{-1} + L_0 + L_1 + L_2 \), where
\[
L_k = L \cap M_8(P)_k.
\] This implies that each \( L_k, k\neq 0 \) is a finitely generated \( \mathcal{A} \)-module and \( S_k^2 L_k = (0) \).

It remains to check that \( [V, [V, L_0]] = (0) \); We will check that \[ [V, [V, M_8(P)_0]] = (0).\] Indeed,
\[
M_8(P)_0 \ \textcolor{.}{\subseteq} \ e_{11}(P) + e_{22}(P) + \textcolor{.}{ e_{55}(P) + e_{66}(P) +}\sum\limits_{\textcolor{.}{\substack{i,j \in \\ \{3,4,7,8\}}}} e_{ij}(P).
\]
Moreover, we have
\[
[V, M_8(P)_0] \subseteq e_{11}(A) + e_{22}(A) \textcolor{.}{+ e_{55}(A) + e_{66}(A) +0,}\] 
\[[V, e_{11}(A) + e_{22}(A) \textcolor{.}{+ e_{55}(A) + e_{66}(A)}] = (0).
\]

We have checked that the superalgebra \( CK(A,d) \) satisfies the assumptions of Chapter \ref{sec:CondFinPresUCE}.

\subsection{Superalgebras \( K(A:n) \), \( n \geqslant 2 \)}\label{subsec:5.4} \mbox{}\\

Let \( n \geqslant 3 \). By Theorem  \ref{theorem:A} the superalgebra \( L = [K(A:n), K(A:n)] \) is perfect and finitely generated. Replace the Grassmann generators \( \xi_1, \xi_2 \) by
\[
\zeta_1 = \frac{1}{\sqrt{2}} (\xi_1 + \sqrt{-1} \xi_2), \quad 
\eta_1 = \frac{1}{\sqrt{2}} (\xi_1 - \sqrt{-1} \xi_2).
\]

Let \( h = \zeta_1 \eta_1 \), \( V = U_{\bar{0}} h \), \( \widetilde{V} = U_{\bar{0}}^2 h \),  \( f(u'h , u''h) = u'u''h; u', u'' \in U_{\bar{0}} \).

Let \( x \) be a product of \( \zeta_1, \eta_1, \xi_3, \ldots, \xi_n \). Let \( a \in A_{\bar{0}}, b \in A, \) and let  \( \deg(x)\) be the degree of \(x\) in the \(\mathbb{Z}\)-grading of $G(n)$. We have
\[
\tag{25} [a \zeta_1 \eta_1, b {x}] =
\begin{cases}
abx, & \text{if } x \text{ involves } \zeta_1, \text{ but does not involve } \eta_1; \\
-abx, & \text{if } x \text{ involves } \eta_1, \text{ but does not involve } \zeta_1; \\
0, & \text{if } x \text{ involves both } \zeta_1 \text{ and } \eta_1; \\
[a, bx] \zeta_1 \eta_1 + a \zeta_1 \eta_1 d(b) x \\ \textcolor{.}{\pm  \tfrac{deg(x)}{2}{\zeta_1\eta_1}d(a)bx} \in A \zeta_1\eta_1 x, & \text{if both } \zeta_1 \text{ and } \eta_1 \text{ are not involved in } x.
\end{cases} 
\]

The Lie algebra \( L \) decomposes into the sum of eigenvalues 
\(
L =  L_{-1} + L_0 + L_1 
\) 
with respect to \( \operatorname{ad}(h).\)

From (25) it follows that \( S_i^2 L_i = 0 \) for any \( i \neq 0 \).

Let \( K(A:n)_0' = A + \sum A \xi_{i_1} \cdots \xi_{i_r} \), where \( 3 \leqslant i_1 < \cdots <i_r \leqslant n \) in each summand. Then
\[
K(A:n)_0 = K(A:n)_0' + \zeta_1\eta_1 K(A:n)_0'.
\]

From (25) it follows that
\[
[A_{\bar{0}}\zeta_1\eta_1, K(A:n)_0] \subseteq \zeta_1\eta_1 K(A:n)_0',
\]
\[
[A_{\bar{0}}\zeta_1\eta_1, \zeta_1\eta_1 K(A:n)_0'] = 0.
\]

This implies the assumption (21).

Now let \( L = K(A:2) \), we assume that \( A = M_2(A) \). Then the superalgebra \( L \) is perfect and finitely generated. The decomposition of \( L \) into a sum of eigenspaces with respect to \( \ad(\zeta_1\eta_1) \) looks as
\[
L = A \eta_1 + (A + A \zeta_1\eta_1)+A \zeta_1.
\]

The inclusions \([V, A + A \zeta_1\eta_1] \subseteq A \zeta_1\eta_1, [V, A \zeta_1\eta_1] = (0)\) imply (21). All other assumptions of Chapter {\ref{sec:CondFinPresUCE}} are straightforward.

\subsection{Twisted superalgebras \(K^{(2)} (A : n), n \geqslant 3\)} \mbox{}\\

Recall that the contact superalgebra \(A\) is equipped with a \(\mathbb{Z}/2\mathbb{Z}\)-grading \(A = A^{(0)} \oplus A^{(1)}\) that is compatible with the superalgebra grading
\[
A = A_{\bar{0}} \oplus A_{\bar{1}}.
\]

Consider the superalgebra \(K(A^{(0)}; n-1)\) generated by \(A^{(0)}, \zeta_1, \eta_1, \xi_3, \ldots, \xi_{n-1}\), 
\[
K^{(2)} (A : n) = K(A^{(0)}; n-1) + \zeta_n K(A^{(0)}; n-1).
\]

As above, we choose \( h = \zeta_1\eta_1 \). \textcolor{.}{Again} let \( U \) be a finite dimensional graded (\textcolor{.}{but} with respect to both gradings) generating subspace of the superalgebra \( A \),  
\[
U = U_{\bar{0}} \oplus U_{\bar{1}} = U^{(0)} \oplus U^{(1)}.
\]

Without loss of generality, we assume that:
\[
U_{\bar{1}} U_{\bar{1}} \subseteq U_{\bar{0}} \quad  \mbox{ and } \quad   U_{\bar{0}}^{(1)}U_{\bar{0}}^{(1)}\subseteq U_{\bar{0}}^{(0)} = U_{\bar{0}} \cap U^{(0)}.
\]

Then \( A \) becomes a finitely generated module over the subalgebra generated by \( U_{\bar{0}}^{(0)} \).  
Let
\[
V = U_{\bar{0}}^{(0)} h, \quad \widetilde{V} = (U_{\bar{0}}^{(0)})^2 h.
\]

The Lie superalgebra \( L = K^{(2)}(A : 2) \) decomposes as
\[
L = L_{-1} + L_0 + L_1, \quad L_i = L \cap K(A : 2)_i, \quad\quad \mbox{ where } i = -1, 0, 1 .
\] 

From (25) it follows that the eigenspaces \( L_{-1}, L_1 \) are finitely generated \(\mathcal{A}\) modules; all other assumptions of Chapter {\ref{sec:CondFinPresUCE} } immediately follow from the fact that they hold in \( K(A:n) \), see Section {\ref{subsec:5.4}}.

\section{Presentation of twisted superconformal algebras} 

In this chapter we separately consider the twisted superalgebra
\[
L = K^{(2)}(F[t,t^{-1}]:2) = F[t^2,t^{-2},\xi_1]  \oplus F[t^2,t^{-2},\xi_1]t\xi_2  
\]
of Ramond type.

The Lie superalgebra \(L\) is generated by elements
\[
t^2, t^{-2}, t^4, t^{-4}, \xi_1, t\xi_2, t\xi_1\xi_2,  t^{-1}\xi_1\xi_2,  , \tag{26}
\]

The elements \(t^{\pm 2}, t^{\pm 4}\) generate the centerless Virasoro algebra \(\Vir(2)\) that is known to be finitely presented (see \cite{fairlie1988}). We have

\begin{align*}
    L = \Vir(2) + \sum_{i \geqslant 0} F \, \ad(t^2)^i (t \xi_1 \xi_2) + 
\sum_{i \geqslant 1} F \, \ad(t^{-2})^i (t \xi_1 \xi_2) + \sum_{i \geqslant 0} F \, \ad(t^2)^i \xi_1  \\ +\sum_{i \geqslant 1} F \, \ad(t^{-2})^i \xi_1 + \sum_{i \geqslant 0} F \, \ad(t^2)^i t \xi_2 + \sum_{i \geqslant 1} F \, \ad(t^{-2})^i t \xi_2. \tag{27}
\end{align*}

Let \( \mathrm{Lie} \langle X \rangle \) be the free Lie superalgebra on \( x_1, \dots, x_8 \). Parities of the free generators \( x_1, \dots, x_8 \) correspond to parities of the generators (26).

Consider the homomorphism \( \mathrm{Lie} \langle X \rangle \xrightarrow{\varphi} L \) that maps the generators \( x_1, \dots, x_8 \) to the generators (26). We call an element \( r(x_1, \dots, x_8) \) a \textit{relation} if \(\varphi(r) = 0\).

We say that an element of \( \mathrm{Lie}\langle X \rangle \) has  degree \( \leqslant n \) if it is a linear combination of commutators in \( x_1, \dots, x_8 \) of length \( \leqslant n \).

\subsection{Relations in   \(L=K^{(2)}(F[t,t^{-1}]:2)\)} \mbox{}\\

Consider the following relations.

\begin{itemize}

\item  Let \(k = \pm 2\) or \(\pm 4\), \(m \geqslant 0\), \(\sigma = 1\) or \(-1\). Suppose that \(k+\sigma m\) has the same sign as \(\sigma\). Then there exist \(s \geqslant 0\) and a nonzero scalar \(\alpha_1(k,\sigma,m) \in F\) such that
\begin{equation} \tag{28}
\operatorname{ad}(t^k) \operatorname{ad}(t^\sigma \xi_1 \xi_2)^m \xi_1 = \alpha_1(k,\sigma,m) \operatorname{ad}(t^\sigma \xi_1 \xi_2)^s \xi_1
\end{equation}
The scalar \(\alpha_1(k,\sigma,m)\) can be found via an appropriate computation in the superalgebra \(L\), but it is irrelevant, \(s = |k+\sigma m|\).

The integer \(k+\sigma m\) and \(\sigma\) have different signs if and only if \(k = \pm 4\), \(\sigma = \mp 1\), \(0 \leqslant m \leqslant 3\) or \(k = \pm 2\), \(\sigma = \mp 1\), \(0 \leqslant m \leqslant 1\). In all cases the relations
\(
\operatorname{ad}(t^k) \operatorname{ad}(t^\sigma \xi_1 \xi_2)^m \xi_1 \in F \operatorname{ad}(t^\sigma \xi_1 \xi_2)^s \xi_1 \text{ have degrees } \leqslant 5.
\)

\item   Similarly, under the same condition there exists a scalar \( \alpha_2 (k, \sigma, n) \in F \)
such that
\[
\ad(t^k)\ad(t^{\sigma}\xi_1\xi_2)^n(t\xi_2) = \alpha_2(k, \sigma, n)\ad(t^{\sigma}\xi_1\xi_2)^s(t\xi_2)
;\tag{29}
\]

\item 
 The relations (28), (29) and relations of degree \( \leqslant 5 \) immediately imply
\[
\ad(t^{2\sigma})^n\xi_1 = \beta_1(\sigma, n)\ad(t^{\sigma}\xi_1\xi_2)^{2n}\xi_1, \quad 0 \neq \beta_1(\sigma, n) \in F, \sigma \pm 1, n \geqslant 0 ;
\tag{30}
\]
\[
\ad(t^{2\sigma})^n(t{\xi_2}) = \beta_2(\sigma, n)\ad(t^{\sigma}\xi_1\xi_2)^{2n}(t{\xi_2}). \tag{31}
\]

\item   For nonnegative integers \(m,n\) and \(\sigma,\tau = \pm 1\) we consider the relations
\[
[\ad(t^{2\sigma})^n(\xi_1\xi_2 t), \ad(t^{2
\tau})^m(\xi_1\xi_2 t)] = 0.
\tag{32}
\]

\item   For nonnegative integers \(m,n\) and \(\sigma,\tau = \pm 1\) we consider the relations \textcolor{.}{(in which \(
0 \neq \gamma_i(n, m, \sigma, \tau) \in F)
\)}
\[
[\ad(t^{\sigma}\xi_1\xi_2)^{2n}\xi_1, \ad(t^{\tau}\xi_1\xi_2)^{2m}\xi_1] = \gamma_1(n, m, \sigma, \tau)t^{2(\sigma{n}+\tau{m})};
\tag{33}
\]
\[
[\ad(t^{\sigma}\xi_1\xi_2)^{2n}(t{\xi_2}), \ad(t^{\tau}\xi_1\xi_2)^{2m}(t{\xi_2})] = \gamma_2(n, m, \sigma, \tau)t^{2(\sigma{n}+\tau{m}+1)}.
\tag{34}
\] For any nonnegative integers \(m,n\) and \(\sigma,\tau = \pm 1\) there exist \(\varepsilon = \pm 1\) and \(s \geqslant 0\) such that
\[
[\ad(t^{\sigma}\xi_1\xi_2)^{2n} \xi_1, \ad(t^{\tau}\xi_1\xi_2)^{2m}(t{\xi_2})] = \gamma_3(n, m, \sigma, \tau)\ad(t^{2\varepsilon})^{s}(t{\xi_1} \xi_2).
\tag{35}
\]

For \( k = \pm 4, \sigma = \pm 1, n \geqslant 0 \) there exist \( \varepsilon = \pm 1, s \geqslant 0 \)
such that
\[ 
[t^k, \ad(t^{2\sigma})^n (t\xi_1\xi_2)] = \mu(k, \sigma, n) \ad(t^{2\varepsilon})^s(t\xi_1\xi_2),
\tag{36}
\]
where \( 0 \neq \mu(k, \sigma, n) \in F. \)
\end{itemize}

\subsection{Presentation by Generators and Relations}\label{sec:5.8} \mbox{}\\
\mbox{}

Consider the Lie superalgebra \(\widetilde{L}\) presented by generators \( x_1, \ldots, x_8 \) and all relations in generators (26) of degree \( \leqslant 5 \). Abusing notation, the images of \( x_i\)'s in \(\widetilde{L}\) \textcolor{.}{are listed by} (26). In other words, from now on we only assume that the generators (26) satisfy all relations \textcolor{.}{(i.e. elements of \(\ker\varphi\))} of degree \( \leqslant 5 \).

\textbf{Claim:} Each of the relations (28)-(36) follows from (i) relations (28)-(36) of smaller degrees;  and (ii) relations of degree \( \leqslant 5 \). Hence, all the relations (28)-(36) hold in the Lie superalgebra \(\widetilde{L}\).

\textit{Note that the defining relations of the Virasoro algebra have degrees \(\leqslant 3\). (\cite{fairlie1988})}  Hence, the elements \(t^{\pm{2}}, t^{\pm{4}}\) generate the subalgebra in \(\widetilde{L}\) that is isomorphic to $\Vir(2)$.

\subsection{Verification of Relations} \mbox{}

We will verify the claim for each of the relations \((28)-(36)\).

\begin{itemize}
    \item Let us start with relations \((28)\). 
    
    Let \(k= \pm 2 \mbox{ or } \pm 4.\) We have \([[t^k, t^{\sigma}\xi_1\xi_2]\), \(t^{\sigma}\xi_1\xi_2] = 0\) (Since this is a relation of degree 3). Let $n\geq 2.$ To avoid relations of degree $\leqslant 5$, we also  assume that $n \geqslant 5.$ We have
\[
\ad(t^k) \ad(t^{\sigma}\xi_1\xi_2)^n \xi_1 = 2\ad(t^{\sigma}\xi_1\xi_2) \ad(t^k) \ad(t^{\sigma}\xi_1\xi_2)^{n-1} \xi_1
\]
\[
- \ad(t^{\sigma}\xi_1\xi_2)^2 \ad(t^k) \ad(t^{\sigma}\xi_1\xi_2)^{n-2} \xi_1 =
\]
\[
(2\alpha(k,\sigma,n-1) - \alpha(k,\sigma,n-2)) \ad(t^{\sigma}\xi_1\xi_2)^s \xi_1
\]

\item In view of the relations \((28)\) of smaller degrees, the claim for the relations \((29)\) is proved similarly.

\item The relations (30), (31) are particular cases of the relations (28), (29). 

\item Consider a relation  (32): using (32) of smaller degrees, we get
\[
\left[ \ad(t^{2\sigma})^n (\xi_1 \xi_2 t), \ad(t^{2\tau})^m (\xi_1 \xi_2 t) \right] = 
(-1)^m \left[ \ad(t^{2\tau})^m \ad(t^{2\sigma})^n (\xi_1 \xi_2 t), \xi_1 \xi_2 t \right].
\]

Suppose that \( \sigma \neq \tau; m \geqslant 1, n \geqslant 1 \). Using
\[
\ad(t^{2\tau}) \ad(t^{2\sigma}) = \ad(t^{2\sigma}) \ad(t^{2\tau}) \pm  4\ad(\mathbbm{1}),
\]
\[
\ad(\mathbbm{1}) \ad(t^k) = \ad(t^k) \ad(\mathbbm{1}) - k \ad(t^k),
\]
we either reduce the length of the product or move \( \ad(t^2) \ad(t^{-2}) \) to the right end and then use
\[
\ad(t^2) \ad(t^{-2}) (\xi_1 \xi_2 t) = -\xi_1 \xi_2 t.
\]

Assume \( m = 0 \) and consider
\[
\left[ \ad(t^2)^n (\xi_1 \xi_2 t), \xi_1 \xi_2 t \right], n \geqslant 2.
\] The case \( \sigma=-1 \) is similar. In the superalgebra \( L \)

\[
\ad(t^2)^n(\xi_1\xi_2 t) = -(2n-1) \ad(t^4) \ad(t^2)^{n-2} (\xi_1\xi_2 t) 
\] is a relation of type (36) of smaller degree. \footnote{\textcolor{.}{The following relation has degree $3$: \(
\ad(t^2)^2 (\xi_1\xi_2 t) = -3 \ad(t^4) (\xi_1\xi_2 t).
\)}}

Now, 
\begin{align*}
\left[ \ad(t^2)^n (\xi_1 \xi_2 t), \xi_1 \xi_2 t \right]&=\left[ \ad(t^2)^{\textcolor{.}{2}}\ad(t^2)^{n-2} (\xi_1 \xi_2 t), \xi_1 \xi_2 t \right]\\
&=\left[ \ad(t^2)^{n-2} (\xi_1 \xi_2 t), \ad(t^2)^2(\xi_1 \xi_2 t) \right]\\
&=-3\left[ \ad(t^2)^{n-2} (\xi_1 \xi_2 t), \ad(t^4)(\xi_1 \xi_2 t) \right].
\end{align*}

On the other hand, 
\begin{align*}
\left[ \ad(t^2)^n (\xi_1 \xi_2 t), \xi_1 \xi_2 t \right]&=-(2n-1)\left[ \ad(t^4)\ad(t^2)^{n-2} (\xi_1 \xi_2 t), \xi_1 \xi_2 t \right]\\
&=(2n-1)\left[\ad(t^2)^{n-2} (\xi_1 \xi_2 t), \ad(t^4)(\xi_1 \xi_2 t) \right], 
\end{align*} which implies \([\ad(t^2)^n(\xi_1\xi_2t), \xi_1\xi_2t] = 0\).

\item Among the relations (33), (34), (35) we will discuss only (33). The relations (34), (35) are treated similarly.

By relations (31) of smaller degree we have
\[
\ad(t^{\sigma}\xi_1\xi_2)^{2n}\xi_1 = \frac{1}{\beta_1(\sigma, n)} \ad(t^{2\sigma})^n\xi_1,
\]
\[
\ad(t^{\tau}\xi_1\xi_2)^{2m}\xi_1 = \frac{1}{\beta_1(\tau, m)} \ad(t^{2\tau})^m\xi_1.
\]

Consider the commutator
\[
[\ad(t^{2\sigma})^n\xi_1, \ad(t^{2\tau})^m\xi_1].
\]

If \(m \geqslant 1\) then this commutator is equal to
\[
[t^{2\tau}, [\ad(t^{2\sigma})^n\xi_1, \ad(t^{2\tau})^{m-1}\xi_1]] - [\ad(t^{2\tau}) \ad(t^{2\sigma})^n\xi_1, \ad(t^{2\tau})^{m-1}\xi_1].
\]

Now,
\[
[\ad(t^{2\sigma})^n\xi_1, \ad(t^{2\tau})^{m-1}\xi_1] =
\]
\[
\beta_1 (\sigma, n), \beta_1 (\tau, m-1) \gamma_1 (n, m-1, \sigma, \tau) t^{2(\sigma{n} + \tau(m - 1))}
\]
by the relations (30), (33) of smaller degree. Hence,
\[
[\ad(t^{2\sigma})^n \xi_1, \ad(t^{2\tau})^m \xi_1] \in (-1)^m [\ad(t^{2\tau})^m \ad(t^{2\sigma})^n \xi_1,\xi_1]
 + F t^{2(\sigma{n} + \tau{m})}.
\]

Suppose that \( \tau \neq \sigma\) and \( m \geqslant 1, n \geqslant 1 \). Arguing 
like we did for relations (32), we move \( \ad(t^2) \ad(t^{-2}) \) to the right side of the product and use
\[
\ad(t^2) \ad(t^{-2}) \xi_1 = -3 \xi_1.
\]

It remains to consider the case \( \sigma = \tau, m=0, n \geqslant \textcolor{.}{3} \): consecutively   applying the relations (30), (28) and again (30) of smaller degrees, we get
\[
\ad(t^{2\sigma})^n \xi_1 = \gamma \ad(t^{4\sigma}) \ad(t^{2\sigma})^{n-2} \xi_1, \gamma \in F.
\]

We also have \( \ad(t^{2\sigma})^2 \xi_1 = \gamma'\ad(t^{4\sigma}) \xi_1 \) , and \( \gamma + \gamma' \neq 0 \).

In view of the above,  
\[[\ad(t^{2\sigma})^n \xi_1, \xi_1] = [\ad(t^{2\sigma})^{n-2} \xi_1, \ad(t^{2\sigma})^2 \xi_1]\]  
\[= \gamma'[\ad(t^{2\sigma})^{n-2} \xi_1, \ad(t^{4\sigma}) \xi_1] \mod Ft^{2{\sigma}n}.\]

On the other hand,  
\[[\ad(t^{2\sigma})^n \xi_1, \xi_1] = r[\ad(t^{4\sigma}) \ad(t^{2\sigma})^n \xi_1, \xi_1]\]
\[= -\gamma[\ad(t^{2\sigma})^n \xi_1, \ad(t^{4\sigma}) \xi_1] \mod Ft^{2{\sigma}n}.\]

This implies  
\[[\ad(t^{2\sigma})^n \xi_1, \xi_1] \in Ft^{2{\sigma}n}.\]

This relation has degree $\leqslant 4$.

\item Now let us consider the relation (36).  
Let \( n = n_1 + n_2 \), \( n_i \geqslant 1 \).  

By (35),  
\[
\ad(t^{2\sigma})^{n} (t \xi_1 \xi_2) = \mu[\ad(t^{\sigma} \xi_1 \xi_2)^{n_1} \xi_1, \ad(t^{\sigma} \xi_1 \xi_2)^{n_2}(t \xi_2)]
\]
\[
0 \neq \mu \in F.
\]

Hence,
\[
[t^k, \ad(t^{2\sigma})^n (t\xi_1\xi_2)] = 
\]
\[
\mu \bigl[[t^k, \ad(t^{\sigma}\xi_1\xi_2)^{n_1}\xi_1],\ad(t^{\sigma}\xi_1\xi_2)^{n_2} (t{\xi_2})\bigr] + 
\]
\[
\textcolor{.}{\mu}\bigl[\ad(t^{\sigma}\xi_1\xi_2)^{n_1}\xi_1, [t^k, \ad(t^{\sigma}\xi_1\xi_2)^{n_2}(t{\xi_2})]\bigr]
\]
\end{itemize}

Now it remains to consequently apply the relations (28), (29), (35); this will complete our proof of the Claim. This results in:




\begin{theorem}\label{thm5.13manuscript}
    The superalgebra \( L = K^{(2)} (F[t,t^{-1}]:2) \) is presented by generators (26) and relations of degree \(\leqslant \textcolor{.}{5}\) .
\end{theorem} 

\textbf{Proof:}  The relations (28) -- (36) imply that the subspace 
\[
\Vir(2) + \sum_{i \geqslant 0} \ad(t^2)^i (t\xi_1\xi_2) + \sum_{i \geqslant 1} \ad(t^{-2})^i (t\xi_1\xi_2)
\]
\[
+ \sum_{i \geqslant 0} F \ad(t^2)^i \xi_1 + \sum_{i \geqslant 1} F \ad(t^{-2})^i \xi_1 + \sum_{i \geqslant 0} F \ad(t^2)^i (t{\xi_2})
\]
\[
+ \sum_{i \geqslant 1}  F \ad(t^{-2})^i (t{\xi_2}),
\]
is a subsuperalgebra in \( \textcolor{.}{\widetilde{L}} \). Since this subsuperalgebra contains generators it follows that it is equal to \( \textcolor{.}{\widetilde{L}} \). Hence the homomorphism \( \textcolor{.}{\widetilde{L}} \rightarrow L \) is an isomorphism. It completes the proof of the theorem.


\section*{Acknowledgments}
The first author has been partially supported by the Spanish Ministry of Science and Universities, project MCIU-22-PID2021-123461NB-C22. The second author has been partially supported by NSFC Grant No. 12350710787 and the Guangdong Program 2023JC10X085.

\bibliographystyle{plain}
\bibliography{refs}

\end{document}